\documentclass{amsart}
\usepackage{amsmath,amssymb,graphicx,bbm,amsthm,subfigure}
\usepackage[american]{babel}
\usepackage{color}
\usepackage[all]{xy}
\newtheorem{lemma}{Lemma}
\newtheorem{theorem}{Theorem}
\newtheorem{proposition}{Proposition}
\newtheorem{claim}{Claim}
\newtheorem{cor}{Corollary}
\newtheorem{obs}{Observation}

\numberwithin{equation}{section}

\DeclareMathOperator{\sys}{sys}
\DeclareMathOperator{\arccosh}{arccosh}
\usepackage{xcolor}
\usepackage{verbatim}

\usepackage{listings}
\usepackage{color} 
\definecolor{mygreen}{RGB}{28,172,0} 
\definecolor{mylilas}{RGB}{170,55,241}

\begin{document}

\lstset{language=Matlab,
    breaklines=true,
    morekeywords={matlab2tikz},
    keywordstyle=\color{blue},
    morekeywords=[2]{1}, keywordstyle=[2]{\color{black}},
    identifierstyle=\color{black},
    stringstyle=\color{mylilas},
    commentstyle=\color{mygreen},
    showstringspaces=false,
    numbers=left,
    numberstyle={\tiny \color{black}},
    numbersep=9pt, 
	xleftmargin=1cm
}

\title{Maximal systole of hyperbolic surface with largest $S^3$ extendable abelian symmetry}

\author{Yue Gao}
\address{BICMR and School of Mathematical Sciences, Peking University, Beijing 100871, CHINA}
\email{yue\_gao@pku.edu.cn}

\author{Jiajun Wang}
\address{School of Mathematical Sciences, Peking University, Beijing 100871, CHINA}
\email{wjiajun@pku.edu.cn}

\date{}

\maketitle

\begin{abstract}
	We give the formula for the maximal systole of the surface admits the largest $S^3$-extendable abelian group symmetry. The result we get is $2\arccosh K$. Here 
\begin{eqnarray*}
		K &=& \sqrt[3]{\frac{1}{216}L^3 +\frac{1}{8} L^2 + \frac{5}{8} L - \frac{1}{8} + \sqrt{\frac{1}{108}L(L^2+18L+27)} } \\
		& & + \sqrt[3]{\frac{1}{216}L^3 +\frac{1}{8} L^2 + \frac{5}{8} L - \frac{1}{8} - \sqrt{\frac{1}{108}L(L^2+18L+27)} } \\
		& & + \frac{L+3}{6}.
\end{eqnarray*}
and  $L= 4\cos^2 \frac{\pi}{g+1}$. 
\end{abstract}

\tableofcontents
\section{Introduction}

  In the study of hyperbolic surfaces, systole is an important topic. 
  It has a wide connection with topics on surfaces, like Teichm\"uller theory and differential geometry on surfaces. On Teichm\"uller theory, from the classical Mumford's compactness criterion \cite{munf} to some recent progress in the Weil-Peterson metric \cite{wolpert17}, \cite{wu19growth}, systole plays an important role. In differential geometry, it relates to  the spectrum of Laplacian operator (\cite{bmm}\cite{bmm18}\cite{mon}), the optimal systolic ratio (\cite{chen15systoles}\cite{ck03universal}\cite{gromov83}) and so on. 
For a survey on the study of the systole, see Parlier \cite{parlier2014simple}. 

    In this paper, systole on a closed hyperbolic surface is refered to either a shortest closed geodesic or its  length, and we often used it to refer to the latter.
	In another point of view, systole is also a function on $\mathcal{M}_g$, the moduli space of all the genus-$g$ closed hyperbolic surfaces.

An important topic on systole is to get the maximal systole on $\mathcal{M}_g$. This question is very hard. Today the only known closed surface with the maximal systole is the Bolza surface with genus 2, which was proved by Jenni \cite{jenni1984ersten}. 
However, there are progresses on globally maximal systole on subspaces of $\mathcal{M}_g$ and locally maximal systoles on $\mathcal{M}_g$. On subspace-maxima,  Bavard got the maximal systole of genus 2 and 5 hyperelliptic surfaces in \cite{bavard1992systole} in 1992. On locally maximal systole, Schmutz \cite{schmutz1993reimann} gave a necessary and sufficient condition for the surfaces with locally maximal systole 
and got some local-maximal-systole examples with polyhedral symmetry. Recently Bourque and Rafi constructed surfaces with locally maximal systoles and trivial symmetry in \cite{bourque18local}. 

Another approach is estimating the lower bound of the maximal systole. 
 P. Buser and P. C. Sarnak \cite{buser1994period} got surfaces with systoles longer than $4/3\log g$ in 1994 by arithmetic method. Later Katz, Schaps and Vishne \cite{katz2007logarithmic} obtained more surfaces with this lower bound. Part of their examples are surfaces with the Hurwitz symmetry. 
 Recently, \cite{petri2015graphs} and \cite{petri2018hyperbolic} obtained more concrete examples with systole longer than $4/7\log g -K$ by different method.

 Inspired by \cite{katz2007logarithmic} and \cite{schmutz1993reimann}, we are interested in the connections between surface symmetry and systoles and consider the maximal systole of surfaces with some big-order symmetries.

The symmetry we consider is the largest $S^3$ extendable abelian symmetry.

 The $S^3$-extendable symmetry on topological surface was recently defined in \cite{wang2013embedding} and \cite{wang2015embedding} and the largest $S^3$-extendable abelian symmetry is a special type of it.

Here is the definition of $S^3$-extendable symmetry on topological surface: For the finite group $G$ acting on the surface $\Sigma_g$, $G$ is $S^3$-extendable if and only if there is an embedding $i:\Sigma_g\to S^3$ such that for any $ g\in G$, there is a $\tilde{g}\in SO(4)$ acting on $S^3$ and the following diagram commutes: 
\[
		\xymatrix{
				\Sigma_g \ar[d]^{i} \ar[r]^{g} & \Sigma_g \ar[d]^{i} \\
				S^3 \ar[r]^{\tilde{g}} & S^3. 
		}
\]

Here is a definition for the largest $S^3$ extendable abelian symmetry: For the $S^3$ extendable surface $(\Sigma,G)$, when $G$ is an abelian group and the order of $G$ is maximal among all the abelian group acting on $\Sigma_g$ $S^3$-extendably, we call $(\Sigma_g,G)$ a genus $g$ surface with largest $S^3$-extendable abelian symmetry.

For convenience, we call $(\Sigma_g,G)$, the genus $g$ largest $S^3$ extendable abelian surface: $\Gamma(2,n)$-surface, as \cite{wang2013embedding} \cite{wang2015embedding}. Here $n=g+1$ . 

We give the definition of hyperbolic version of this symmetry: 
For a hyperbolic surface $\Sigma_g$ with a group $G$ isometrically acting on it, if $(\Sigma_g,G)$ is a surface with largest $S^3$ extendable abelian symmetry when forgeting the hyperbolic structure, then we call $(\Sigma_g,G)$ a hyperbolic surface with largest $S^3$ extendable abelian symmetry. 
All of the genus $g$ hyperbolic surfaces with largest $S^3$ extendable abelian symmetry form a subspace of $\mathcal{M}_g$, which can be parametrized by two parameters (for details, see Section \ref{sec_construct}). It is known that systole can be viewed as a function from $\mathcal{M}_g$ to $\mathbb{R}^+$. 
We define the {\em maximal systole of $\Gamma(2,n)$ surface }to be maximal value of the systole function on the subspace of $\mathcal{M}_g$ of all the genus $g$ hyperbolic surfaces with largest $S^3$ extendable abelian symmetry.

Our result is:

 \begin{theorem}
		 The maximal systole of the $\Gamma(2,n)$ surface is 
		 \[
				 2\arccosh K. 
		 \]

Here 
\begin{eqnarray*}
		K &=& \sqrt[3]{\frac{1}{216}L^3 +\frac{1}{8} L^2 + \frac{5}{8} L - \frac{1}{8} + \sqrt{\frac{1}{108}L(L^2+18L+27)} } \\
		& & + \sqrt[3]{\frac{1}{216}L^3 +\frac{1}{8} L^2 + \frac{5}{8} L - \frac{1}{8} - \sqrt{\frac{1}{108}L(L^2+18L+27)} } \\
		& & + \frac{L+3}{6}.
\end{eqnarray*}
and  $L= 4\cos^2 \frac{\pi}{n}$. 

 The maximal systole is obtained when 
 \[
		 (c,t) = \left ( \arccosh K, 2\arccosh \frac{K + 1}{2 \cos \frac{\pi}{n}} \right ).  
 \]
 The symbol $c$ and $t$ are defined in Section \ref{sec_construct}. 

		 \label{thm_main}
 \end{theorem}

 Numerical results of small genera are shown in Table \ref{tab_num}. 

 Compared with the work \cite{bgw} on the systole of surface with large cyclic symmetry, the surface with large cyclic symmetry has unique geometric structure, while surfaces studied in this paper form a subspace of $\mathcal{M}_g$ parametrized by two parameters when $g\ge2$. Hence the method in this paper is quite different from that in \cite{bgw}. In this paper, we obtain our result by classifying the family of curves that are possible to be the shortest geodesics on surfaces admitting this symmetry and then getting the condition when the systole is maximal.

 \begin{table}
		 \centering
		 \begin{tabular}{ll}
				 Genus & $\sys = 2\arccosh K $ \\

2	&     3.0571\\
3	&     3.6478\\
4	&     3.9078\\
5	&     4.0464\\
6	&     4.1291  
		 \end{tabular}
		 \caption{Maximal systole of surface with largest $S^3$-extendable abelian symmetry}
		 \label{tab_num}
 \end{table}

 It is worth to note that when $g = 2 $, the $\Gamma(2,n)$ surface is the Bolza surface. Bolza surface is the genus 2 surface with the maximal systole and is constructed by attaching the opposite sides of a hyperbolic regular octagon (Figure \ref{fig_bolza_2}). In Figure \ref{fig_bolza_1}, the pants is one of the two pants that forms the $\Gamma(2,3)$ surface. The points $A_1, \dots, A_6$ are fixed points of the hyperelliptic involution. All of the segments between these points shown in Figure \ref{fig_bolza_1} have the same length (actually half of the systole). There is an isometry between $\Gamma(2,3)$ surface (Figure \ref{fig_bolza_1}) and the Bolza surface (Figure \ref{fig_bolza_2}). This isometry is shown in Figure \ref{fig_bolza}: the $A_i$ in Figure \ref{fig_bolza_1} are mapped to Figure \ref{fig_bolza_2}. 

\begin{figure}[htbp]
		\centering
		\subfigure[]{
		\includegraphics{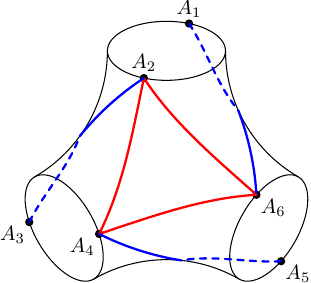}
		\label{fig_bolza_1}
}
\subfigure[]{
		\includegraphics{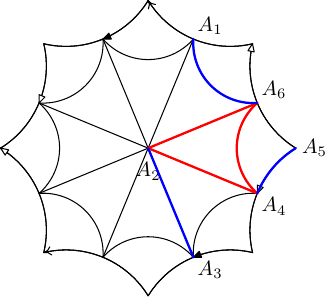}
		\label{fig_bolza_2}
}
		\caption{}
		\label{fig_bolza}
\end{figure}

In Section \ref{sec_construct}, we construct the hyperbolic $\Gamma(2,n)$ surface and describe its symmetry. In Section \ref{sec_int}, we give some useful lemmas on the intersection of the systoles. 

 Then in Section \ref{sbsc_image}, by the hyperellipticity of the $\Gamma(2,n)$ surface, we proved that systoles must be the simple curves meeting the singular points of the orbifolds (the quotients of the surface by the symmetric group). 

 We divided such curves into three types in Section \ref{sbsc_image} and found the shortest one in each type in Section \ref{sec_shortest}. These are the central parts of this work. 

 Next we prove that the surface's systole is maximal if and only if the three curves have the same length (Proposition \ref{prop_equal}). Finally we calculate this length using the condition that the three candidates have equal length. 

We give more details of the part dividing the curves into three types. 
We divided the curves in the orbifold by the singular points they meet. 
In each type, a curve corresponds to an element in $\mathbb{Z}$ (Lemma \ref{lem_corresp_seg} and \ref{lem_corresp_scc}). 
Then based on this characterization, by the method of constructing cyclic covers similar to that using Seifert surface in knot theory, we construct the surface from orbifold and prove that for two curves $l$, $l'$ in the orbifold of the same type, if $l$ is longer than $l'$, then $l$'s lift is longer than $l'$'s (Theorem \ref{thm_equ}). Then in each type, we found the shortest curve in Proposition \ref{prop_reduce} together with Proposition \ref{prop_cuff_fake_not} and \ref{prop_11_not}. (Proposition \ref{prop_cuff_fake_not} and \ref{prop_11_not} exclude two curves that cannot be lifted to a systole in the maximal surface and are not excluded by Proposition \ref{prop_reduce}. )

{\bf Acknowledgement: }
We acknowledge Prof. Shicheng Wang and Sheng Bai for many helpful discussions and suggestions. Yue Gao acknowledge Prof. Shicheng Wang's supervison and many help and support during his PhD studies. We acknowledge Prof. Ying Zhang for helpful comments.

\section{The construction and symmetry of $\Gamma(2,n)$ surface}
\label{sec_construct}

We construct the hyperbolic $\Gamma(n,2)$ surface, which is similar to the construction of topological $\Gamma(n,2)$ surface in \cite{wang2015embedding} and \cite{wang2013embedding}. 

We pick two isometric $n$-holed spheres with the order $n$ rotation symmetry. Similarly to $3$-holed spheres. we call a boundary component of the $n$-holed spheres a {\em cuff} and the common perpendicular between two neighboring cuffs a {\em seam}. (See Figure \ref{fig_exp_1_1}.) The two spheres are attached along their boundary components in the order shown in Figure \ref{fig_exp_1_1}. The twist parameter at all the boundary components are the same to give the $\Gamma(2,n)$ surface the order $n$ rotation symmetry frome the $n$-holed spheres. 

\begin{figure}[htbp]
		\centering
		\includegraphics{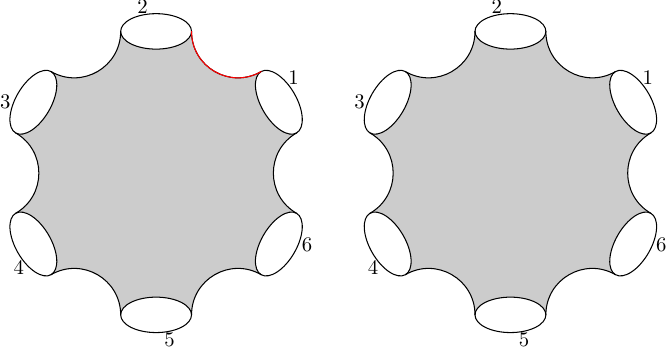}
		\caption{$\Gamma(2,n)$ surface. The red segment is a seam. Two cuffs with the same label are attached together. }
		\label{fig_exp_1_1}
\end{figure}

\subsection{The symmetry of $\Gamma(2,n)$ surface} 

The symmetric group of the $\Gamma(2,n)$ surface is $\mathbb{Z}^2\oplus \mathbb{Z}^n$. We denote the order $n$ and order $2$ generators of the group by $\sigma$ and $\tau$ respectively. $\sigma$ maps each $n$-holed sphere to itself, and on each $n$-holed sphere, $\sigma$ is the order $n$ rotation, while $\tau$ is the order two rotation exchanges the two $n$-holed spheres. 

Assume $\Sigma$ is a $\Gamma(2,n)$ surface, then $\Sigma/\langle\tau\rangle = S^2(2,2,\dots,2,2)$, and $(\Sigma/\langle\tau\rangle)/\langle\sigma\rangle = S^2(2,2,n,n)$. 

We assume 
\[
		\pi:\Sigma\to\Sigma/\langle\tau\rangle
\]
to be the double cover induced by $\tau$ and 
\[
		\pi':\Sigma/\langle\tau\rangle\to(\Sigma/\langle\tau\rangle)/\langle\sigma\rangle
\]
is the branch cover induced by $\sigma$ on the orbifold. 

Here is the diagram of the covering: 

\begin{equation}
		\xymatrix{
				\Sigma \ar[r]^{\pi} && S^2(2,2,\dots,2) \ar[r]^{\pi'}& S^2(2,2,n,n)
		}. 
		\label{for_corver_pre}
\end{equation}

On the orbifold $S^2(2,2,n,n)$, there is a branched double covering $\pi''$ from $S^2(2,2,n,n)$ to $S^2(2,2,2,n)$ (See Figure \ref{fig_22nn_222n_1}). 
\begin{figure}[htbp]
		\centering
		\includegraphics{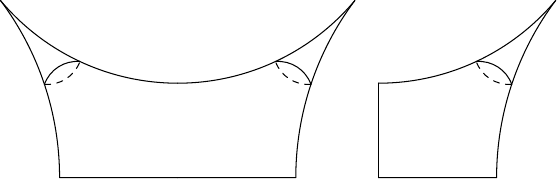}
		\caption{}
		\label{fig_22nn_222n_1}
\end{figure}
Then we can extend the Diagram (\ref{for_corver_pre}) longer:
\begin{equation}
		\xymatrix{
				\Sigma \ar[r]^{\pi} && S^2(2,2,\dots,2) \ar[r]^{\pi'}& S^2(2,2,n,n)\ar[r]^{\pi''} & S^2(2,2,2,n)
		}. 
		\label{for_corver_pre_2}
\end{equation}

\subsection{The geometry of $\Gamma(2,n)$ surface} 
\label{sec_geo_orbi}

The $\Gamma(2,n)$ surface consists of two isometric $n$-holed spheres (Figure \ref{fig_exp_1_1}). The geometric structure of the surface is determined by two parameters similar to the Fenchel-Nielsen parameters.
One of the parameters is the length of the cuffs (denoted $2c$), 
the other is the ``twist parameter'', the following is its formal definition. This parameter (denoted $t$) is the distance between two end points of seams on a cuff. The two seams are required to be in different $n$-holed spheres, and both seams' end points are in the same two cuffs. By this definition, $t\le 2c/2=c$. (See Figure \ref{fig_twist}. In this Figure, the circle $c_1$, $c_2$ are cuffs of a $\Gamma(2,n)$ surface, $AA'$ is a seam in one $n$-holed spheres and $BB'$ is a seam in the other. Then $t$ is the distance between the points $A$ and $B$ on the cuff)

\begin{figure}[htbp]
		\centering
		\includegraphics{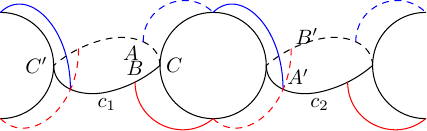}
		\caption{}
		\label{fig_twist}
\end{figure}

The length of the seams is denoted by $s$. We give a relation between $s$ and $c$ that will be used in latter calculation. Each $n$-holed sphere of the $\Gamma(2,n)$ surface consists of two isometric right-angled $2n$-polygons with order $n$ rotation symmetry. By the definition of seam and cuff, the edge length of the $2n$-polygon is $s$ or $c$. If an edge's length is $s$, then its neighboring edges' length is $c$; If an edge's length is $c$, then its neighboring edges' length is $s$. 

By connecting the center of the polygon and the mid-points of two neighboring edges, we obtained a trirectangle $OABC$ (Figure \ref{fig_n_gon}) with $AB=s/2$, $BC = c/2$, $\angle O = \pi/n$ and $\angle A = \angle B =\angle C =\pi/2$. Then by formula 2.3.1(i) in \cite[p.471]{buser2010geometry}, we have
\begin{equation}
		\cosh \frac{s}{2} \cosh \frac{c}{2} = \cos \frac{\pi}{n}.
		\label{for_cs}
\end{equation}

\begin{figure}[htbp]
		\centering
		\includegraphics{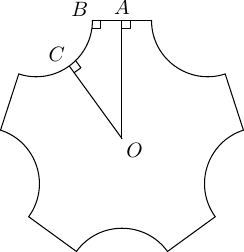}
		\caption{}
		\label{fig_n_gon}
\end{figure}

Now we begin to describe the geometric structures of the orbifolds $S^2(2,2,\dots,2)$, $S^2(2,2,n,n)$ and $S^2(2,2,2,n)$ induced from the $\Gamma(2,n)$ surface. 

An $n$-holed sphere in the construction of the $\Gamma(2,n)$ surface is a fundamental domain of the double branched cover $\pi$. On each cuff, there is a pair of antipodal points that are the fixed points of $\pi$. Therefore the image of each cuff of the $\Gamma(2,n)$ surface is a segment on $S^2(2,2,\dots,2)$ with length $c$. End points of the segment are the singular points. For the seams, since a seam is fully contained in a fundamental domain, the image of a seam is the common perpendicular of two neighboring (images of) cuffs on $S^2(2,2,\dots,2)$ with length $s$. Finally we consider the twist parameter $t$. Since $\pi$ is a locally-isometric double cover, the two seams in $\Gamma(2,n)$ surface connecting the same two cuffs are mapped to the same curve in $S^2(2,2,\dots,2)$. (In Figure \ref{fig_twist}, $AA'$ and $BB'$ are the two seams connecting the same two cuffs, they are mapped to the same curve in $S^2(2,2,\dots,2)$ by $\pi$. ) Therefore, the mid-point between the roots of the two seams ($A$, $B$ in Figure \ref{fig_twist}) is the singular point of $\pi$. We let $C$ in Figure \ref{fig_twist} be the mid-point between $A$ and $B$. Then $AC=BC = t/2$. 

Figure \ref{fig_orbi_geo_1} is the orbifold $S^2(2,2,\dots,2)$ with the geometric structure induced by the $\Gamma(2,n)$ surface and $\pi$. Here $AA'$ is the image of the seam, $|AA'|=s$. $C$, $C'$ (and other blue points) are singular points. The segment $CC'$ is the image of a cuff and $|CC'| = c$. Then $|AC| = t/2$. 

Next we consider the orbifold $S^2(2,2,n,n)$. In Figure \ref{fig_orbi_geo_1}, the four regions seperated by the red lines are the fundamental domains of the $n$-covering map $\pi'$. Here $O$ and $O'$ are the centers of the regular $2n$-polygon, $D$ is the mid-point of $AA'$. $ODO'$ is a red line and other red lines are $ODO'$'s image by the order-$n$ rotation. 

\begin{figure}[htbp]
		\centering
		\includegraphics{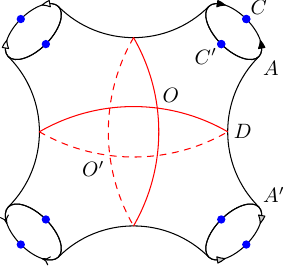}
		\caption{}
		\label{fig_orbi_geo_1}
\end{figure}

Figure \ref{fig_orbi_geo_2} is the $S^2(2,2,n,n)$ orbifold with geometric structure induced from $\Gamma(2,n)$ surface by covering map $\pi$ and $\pi'$. In Figure \ref{fig_orbi_geo_2}, the point $A$ is the image of the $A$ in Figure \ref{fig_orbi_geo_1} and so are other points labelled by the same letter. Thus $C$ and $C'$ are the order $2$ singular points, $O$ and $O'$ are the order $n$ singular points. Then $CC'$ is the image of the cuffs and $|CC'| = c$. $AD$ is the image of the seams and $|AD| = s/2$. And therefore $|AC| = t/2$. 
\begin{figure}[htbp]
\centering
\begin{minipage}[t]{0.45\textwidth}

		\centering
		\includegraphics{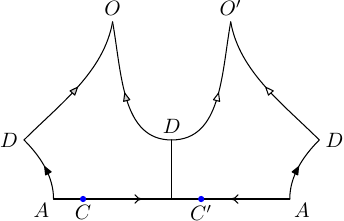}
		
		\caption{\label{fig_orbi_geo_2}}
\end{minipage}
\begin{minipage}[t]{0.45\textwidth}
		\centering
		\includegraphics{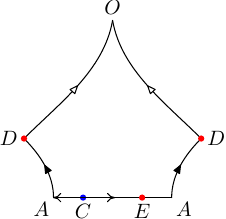}
		\caption{}
		\label{fig_orbi_geo_3}
\end{minipage}
\end{figure}

Finally we consider the orbifold $S^2(2,2,2,n)$. In Figure \ref{fig_orbi_geo_2}, one of the two pentagons is a fundamental domains of the double-covering map $\pi''$. The two fixed point of the covering map are $D$ and the mid-point of $CC'$.

Figure \ref{fig_orbi_geo_3} is the $S^2(2,2,2,n)$ orbifold with geometric structure induced from $\Gamma(2,n)$ surface by covering map $\pi$, $\pi'$ and $\pi''$. In Figure \ref{fig_orbi_geo_3}, the point $A$ is the image of the $A$ in Figure \ref{fig_orbi_geo_2} and so are other points labelled by the same letter except the point $E$. $E$ is the image of the mid-point of $CC'$ of Figure \ref{fig_orbi_geo_2}. $C$, $D$ and $E$ are the order $2$ singular points, $O$ is the order $n$ singular point. Then $CE$ is the image of the cuffs and $|CE| = c/2$. $AD$ is the image of the seams and $|AD| = s/2$. And therefore $|AC| = t/2$. 

\section{The intersections of the systoles}
\label{sec_int}

Here is a basic and very useful claim:

\begin{claim}
		Two systoles intersect at most once. 
		\label{claim_basic}
\end{claim}

A key observation is crucial in proving the three Lemmas in this Section:

\begin{obs}
		Any simple closed curve on $S^2(2,2,\dots,2)$ (or $S^2(2,2,n,n)$ or $S^2(2,2,2,n)$) is seperating. Therefore if two simple closed curves on $S^2(2,2,\dots,2)$ (or $S^2(2,2,n,n)$ or $S^2(2,2,2,n)$) intersect each other, then they intersect at least twice. 
		\label{obs_key}
\end{obs}

Claim \ref{claim_basic} and Observation \ref{obs_key} are essential tools for proving the following three Lemmas. These Lemmas are important tools to characterize the systoles on $\Gamma(2,n)$ surface.

\begin{lemma}
		If $\Sigma$ is a $\Gamma(2,n)$ surface, then the image of a systole of $\Sigma$ on $S^2(2,2,\dots,2)$ by $\pi$ does not intersect itself at any regular point of the orbifold. 

		The images of two systoles of $\Sigma$ on $S^2(2,2,\dots,2)$ by $\pi$ do not intersect at any regular point of the orbifold. 
		\label{lem_hyp_ell}
\end{lemma}

\begin{lemma}
		If $\Sigma$ is a $\Gamma(2,n)$ surface, then the image of a systole of $\Sigma$ on $S^2(2,2,n,n)$ by $\pi'\circ\pi$ does not intersect itself at any regular point of the orbifold. 

		The images of two systoles of $\Sigma$ on $S^2(2,2,n,n)$ by $\pi'\circ\pi$ do not intersect at any regular point of the orbifold. 
		\label{lem_hyp_ell_n}
\end{lemma}

\begin{lemma}
		If $\Sigma$ is a $\Gamma(2,n)$ surface, then the image of a systole of $\Sigma$ on $S^2(2,2,2,n)$ by $(\pi''\circ\pi'\circ\pi$ does not intersect itself at any regular point of the orbifold. 

		The images of two systoles of $\Sigma$ on $S^2(2,2,2,n)$ by $\pi''\circ\pi'\circ\pi(\Sigma)$ do not intersect at any regular point of the orbifold. 
		\label{lem_hyp_ell_n_2}
\end{lemma}

The idea to prove these Lemmas is direct: using Claim \ref{claim_basic}. We assume that a systole's image on the orbifold has self-intersection or two systoles' images intersect each other at regular points. Then we prove that the lift of the images always contain two simple clsoed curve with equal length and at least two intersections. 

But the proofs are rather long, because we need to deal with all the possible shape of the images of a simple closed curve on $\Sigma$ by $\pi$ and the lifts of the curves' images case by case.

\begin{proof}[Proof of Lemma \ref{lem_hyp_ell}]
If $\alpha$ is a simple closed curve in $\Sigma$, $\pi(\alpha)$ has a self-intersection point $p$. Then $\pi^{-1}(p)$ consists of two points, both are the intersection points of $\pi^{-1}(\pi(\alpha))$. By the definition of double cover, $\pi^{-1}(\pi(\alpha))$ consists of either one curve or two curves with equal length. Since $\alpha$ is simple, $\pi^{-1}(\pi(\alpha))$ consists of two curves. These two curves intersect at least twice, therefore cannot be systole. 

We assume $\alpha$ and $\beta$ are two simple closed curves with equal length on $\Sigma$, $p$ is the intersection point of $\pi(\alpha)$ and $\pi(\beta)$. 

We recall that there are two type of order 2 isometric action on $S^1$. Therefore, the shape of $\pi(\alpha)$ and $\pi(\beta)$ has 2 possibilities: $S^1$ or a segment, whose endpoints are fixed points of $\pi$.

(a) If $\pi(\alpha)$ and $\pi(\beta)$ are simple closed curves, then $\pi(\alpha)$ intersects $\pi(\beta)$ at least twice by Observation \ref{obs_key}. 
Recall that there are two types of double cover of $S^1$, namely $S^1$ and $S^1\coprod S^1$. 
Then there are three types of double cover of $\pi(\alpha)\cup \pi(\beta)$, shown in Figure \ref{fig_double_cover}. 

In all the cases, $\alpha$ intersects $\beta$ at least twice, which contradicts to Claim \ref{claim_basic}. 

\begin{figure}[htbp]
		\centering
		\includegraphics{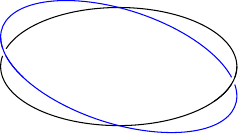}
		\caption{$\pi(\alpha)\cup\pi(\beta)$}
		\label{fig_double_cover_1}
\end{figure}

\begin{figure}[htbp]
		\centering
		\subfigure[]{
		\includegraphics{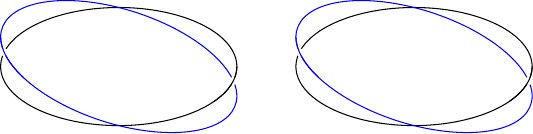}
}

\subfigure[]{
		\includegraphics{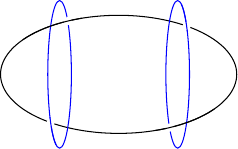}
}
\subfigure[]{
		\includegraphics{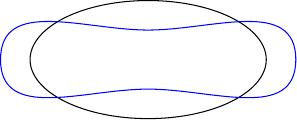}
}
		\caption{The double cover of $\pi(\alpha)\cup\pi(\beta)$}
		\label{fig_double_cover}
\end{figure}

(b) If $\pi(\alpha)$ is a segment while $\pi(\beta)$ is a simple closed curve (Figure \ref{fig_double_cover_1s}), then there are two types of the double cover of $\pi(\alpha)\cup\pi(\beta)$ shown in Figure \ref{fig_double_cover_1s_d}. 

\begin{figure}[htbp]
		\centering
		\includegraphics{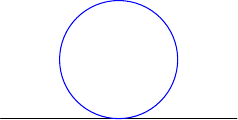}
		\caption{$\pi(\alpha)\cup\pi(\beta)$ (2)}
		\label{fig_double_cover_1s}
\end{figure}

\begin{figure}[htbp]
		\centering
		\subfigure[]{
		\includegraphics{double_cover-1.pdf}
		\label{fig_double_cover_1s_1}
		}
		\subfigure[]{
		\includegraphics{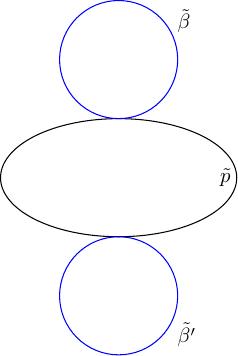}
		\label{fig_double_cover_1s_2}
		}
		\caption{The double cover of $\pi(\alpha)\cup\pi(\beta)$ (2)}
		\label{fig_double_cover_1s_d}
\end{figure}

If the double cover of $\pi(\alpha)\cup\pi(\beta)$ is the case shown in Figure \ref{fig_double_cover_1s_1}, then it is clear that the curve $\alpha$ and $\beta$ have at least two intersctions. Therefore, $\alpha$ and $\beta$ cannot be systoles.

If the double cover of $\pi(\alpha)\cup\pi(\beta)$ is the case shown in Figure \ref{fig_double_cover_1s_2}, we assume $\tilde{p}$ is one of the fix point of $\pi$ in Figure \ref{fig_double_cover_1s_2}. Therefore $\pi_*([\tilde{\beta}]) = \pi_*([\tilde{\beta'}])$ in $\pi_1(\pi(\Sigma), \pi(\tilde{p}))$. Here $[\tilde{\beta}]$ and $[\tilde{\beta'}]$ are elements of $\pi_1(\Sigma,\tilde{p})$ represented by $\tilde{\beta}$ and $\tilde{\beta'}$. It contradicts to the injectivity of $\pi_*$ ($\pi$ is a covering map). 

(c)  If both $\pi(\alpha)$ and $\pi(\beta)$ are segments (Figure \ref{fig_double_cover_2s}). then $|\pi^{-1}(\pi(\alpha))\cap \pi^{-1}(\pi(\beta))|\ge 2$ since the intersction point of $\pi(\alpha)$ and $\pi(\beta)$ is a regular point. However, both $\pi^{-1}(\pi(\alpha))$ and $ \pi^{-1}(\pi(\beta))$ are connected. Therefore $|a\cap b| \ge 2$, so that $\alpha$ and $\beta$ cannot be systole. 
\begin{figure}[htbp]
		\centering
		\includegraphics{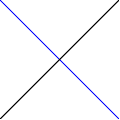}
		\caption{}
		\label{fig_double_cover_2s}
\end{figure}
\end{proof}

Furthermore, we have

\begin{proof}[Proof of Lemma \ref{lem_hyp_ell_n}]
		Recall that $\pi': S^2(2,2,\dots,2)\to S^2(2,2,n,n)$ is the covering map. If $\alpha$ is a systole of $\Sigma$ then by Lemma \ref{lem_hyp_ell}, $\pi(\alpha)$ has no self-intersection and won't intersect the image of another systole at regular points. Therefore, if $\pi'\pi(\alpha)$ has self-intersection at regular points, then it implies that either $\pi(\alpha)$ intersects itself or it intersects another lift of $\pi'\pi(\alpha)$. Therefore $\pi'\pi(\alpha)$ has no self-intersections.

		By exactly the same argument, we can prove that the images of two systoles of $\Sigma$ on $S^2(2,2,n,n)$ do not intersect at any regular point of the orbifold. 

\end{proof}

We omit the proof for Lemma \ref{lem_hyp_ell_n_2}, because it is exactly the same to the proof for Lemma \ref{lem_hyp_ell_n}.

If $\alpha$ is a systole on $\Sigma$ then $\pi''\circ\pi'\circ\pi(\alpha)$ is either a simple closed curve or a segment connecting two singular points. 

Now we begin to characterize and classify the image of systoles on $S^2(2,2,2,n)$. 

\section{The image of systoles on $S^2(2,2,2,n)$}
\label{sbsc_image}

\begin{lemma}
		The image of systoles won't pass the order $n$ singular point. 
		\label{lem_order_n}
\end{lemma}

\begin{proof}
		The order $n$ singular point is lifted to a regular point in $S^2(2,2,\dots,2)$, and a segment in the neighborhood of the order $n$ point that passes the point is lifted to $n$ segments intersecting at the pre-image of the order $n$ singular point. Then by Lemma \ref{lem_hyp_ell}, the curve passing the order $n$ singular point cannot be lifted to a systole in the surface. 
\end{proof}

By Lemma \ref{lem_hyp_ell_n_2}, the image of systoles in $S^2(2,2,2,n)$ is either a simple closed curve or the double of a segment connecting two singular points. We first consider the systole whose image is a segment connecting two singular points.

\begin{lemma}
		For two given singular points $p,q$ of the orbifold $S^2(p,q,r,s)$, there is a 1-1 correspondence between the segment between $p,q$ (up to homotopy) and the elements of the fundamental group $\pi_1(S^1)$. 
		\label{lem_corresp_seg}
\end{lemma}
\begin{proof}
		
		We construct the correspondence directly. 
We pick the fundamental group $\pi_1(|S^2(p,q,r,s)|\backslash\{r,s\})$, which is isomorphic to $\pi_1(S^1)$. We notice that the segment between $p$ and $q$ won't pass $r$ or $s$. Choosing a segment $\alpha$ between $p$ and $q$, we have that for any segment $\beta$ between $p$ and $q$, $\alpha\beta^{-1}$ represents an element of $\pi_1(|S^2(p,q,r,s)|\backslash\{r,s\})$. This is the 1-1 correspondence. 
\end{proof}

We define a symbol $\tilde{l}_{pq}$, to be the family of segments connecting the singular points $p$ and $q$. Figure \ref{fig_class} are the three families on $S^2(2,2,2,n)$. 

\begin{figure}[htbp]
		\centering
		\subfigure[a curve in $\tilde{l}_{DE}$]{
		\includegraphics{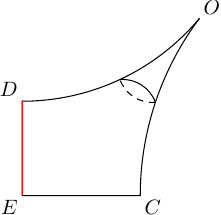}
		\label{fig_class_1}
}
		\subfigure[a curve in $\tilde{l}_{CE}$]{
		\includegraphics{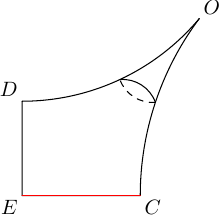}
		\label{fig_class_2}
}
		\subfigure[a curve in $\tilde{l}_{CD}$]{
		\includegraphics{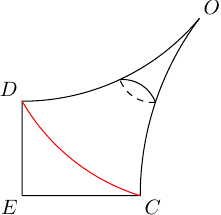}
		\label{fig_class_3}
}
		\caption{}
		\label{fig_class}
\end{figure}

Now we begin to discuss the simple closed curve in $S^2(2,2,2,n)$. 

\begin{lemma}
	The simple closed curve in $S^2(2,2,2,n)$ is the curve 
	contains exactly one order-two singular point and the other two oder-two singular points are in the same side of the curve. 
	(see Figure \ref{fig_222nscc_1}). 	
		\label{lem_unique}
\end{lemma}

\begin{figure}[htbp]
\centering
\begin{minipage}[t]{0.45\textwidth}
		\centering
				\includegraphics{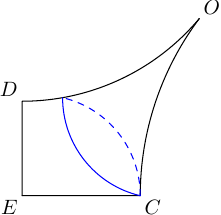}
				\caption{}
				\label{fig_222nscc_1}
\end{minipage}
\begin{minipage}[t]{0.45\textwidth}
		\centering
				\includegraphics{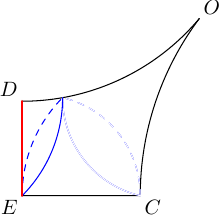}
				\caption{}
				\label{fig_222nscc_2}
\end{minipage}
\end{figure}
\begin{proof}
		First we consider the curve with no singular point on it. Any such simple closed curve in $S^2(2,2,2,n)$ is splitting and in each side there are two singular points (Otherwise the curve corresponds to a torsion element in the fundamental group and cannot be lifted to the surface. ). Moreover, there are two order 2 singular points in one side of the curve. However, the curve cannot be realized as a a geodesic in any hyperbolic structure of the orbifold $S^2(2,2,2,n)$, because the disk it bounds $D^2(2,2)$ has Euler characteristic $0$. If $\partial D^2(2,2)$ is a geodesic boundary, then the interior of $D^2(2,2)$ admits a complete hyperbolic structure, which contradicts to $\chi(D^2(2,2))=0$. 

		If the curve contains one singular point, the only thing to prove is that the other two order-two singular points are in the same side of the curve. Otherwise, there is an order-$n$ singular point and order-$2$ singular point on one side of the curve, while only one order-$2$ singular point on the other side of the curve. 
		See Figure \ref{fig_222nscc_2}, for the blue curve passing the singular point $E$ (denoted $\alpha$), the singular points $O$ (order $n$) and $C$ (order $2$) are on one side of $\alpha$, while the singular point $D$ (order $2$) is on the other side. We assume $\beta$ is a segment connecting $D$ and $E$, and $\beta$ does not intersect $\alpha$ except at $E$. $\alpha$ is homotopic to the double of $\beta$ by the contractibility of a disk and $\beta$ is homotopic to a geodesic connecting $D$ and $E$. By the uniqueness of geodesics in homotopy class, $\alpha$ cannot be a geodesic and then cannot be the image of a geodesic.

		The geodesic passing a order-$2$ singular points always returns (see Figure \ref{fig_return}). 
		\begin{figure}[htbp]
				\centering
				\subfigure[The lift of a geodesic passing an order-$2$ singular point]{
				\includegraphics{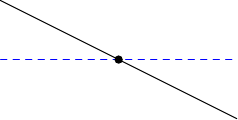}
		}
				\subfigure[The geodesic passing an order-$2$ singular point]{
				\includegraphics{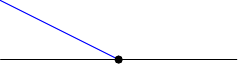}
		}
				\caption{}
				\label{fig_return}
		\end{figure}
		Therefore the image of a systole that contains at least two singular points must be the segment connecting two singular points. The lemma is proved. 

\end{proof}

\begin{lemma}
		There is an injective map from the simple closed curve passing one order-two singular point to the fundamental group $\pi_1(S^1)$. 
		\label{lem_corresp_scc}
\end{lemma}

\begin{proof}
		By Lemma \ref{lem_unique}, for a simple closed curve passing one order-two singular point, the other two order-two singular points are on the same side of the curve. 
		There is a unique segment connecting the two singular points, not meeting the curve. Then by Lemma \ref{lem_corresp_seg}, there is a 1-1 correspondence between segments connecting the two singular points and elements in $\pi_1(S^1)$. This Lemma holds. 

		This segment is unique because the part of the orbifold bounded by the curve that contains the two singular points is a disk (with two singular points), and therefore any two segments connecting the two singular points in this disk are homotopic. 

		The map is injective because by cutting along the segment, we get a disk with two singular points. The curve passing one given singular point is unique up to homotopy.

\end{proof}

We define another symbol $\tilde{l}_p$ to be the family of simple closed geodesics in $S^2(2,2,2,n)$ that pass an order-two singular point $p$. Figure \ref{fig_222nscc_1} is a curve in $\tilde{l}_C$.

We define a symbol for the comvenience to state the following useful theorem. 

We let $l$ be a curve in $\tilde{l}_{pq}$ or $\tilde{l}_p$. We define $|l(S^2(2,2,n,n))|$ be the length of a component of $l$'s preimage in $S^2(2,2,n,n)$. Similarly, we can define $|l(S^2(2,2,\dots,2))|$ and $|l(\Sigma)|$. This definition is well defined by the symmetry of the surface. The detail is in the proof of the following Theorem. 

\begin{theorem}
		For $l, l'\in \tilde{l}_{pq}$ or $\tilde{l}_p$, if $l$ and $l'$ can be lifted to a systole of the $\Gamma(2,n)$ surface, then 
		\begin{eqnarray*}
				\frac{|l(S^2(2,2,n,n))|}{|l|} &=&  \frac{|l'(S^2(2,2,n,n))|}{|l'|} \\
				\frac{|l(S^2(2,2,\dots,2))|}{|l|} &=&  \frac{|l'(S^2(2,2,\dots,2))|}{|l'|} \\
				\frac{|l(\Sigma)|}{|l|} &=&  \frac{|l'(\Sigma)|}{|l'|}. 
		\end{eqnarray*}

		\label{thm_equ}
\end{theorem}

This Theorem has a direct Corollary: 
\begin{cor}
		If $l_0$ is the shortest curve in one curve family ($\tilde{l}_{pq}$ or $\tilde{l}_p$) in $S^2(2,2,2,n)$, then $l_0$'s lift in the $\Gamma(2,n)$ surface $\Sigma$ is not longer than the lift of other curves in the same family. 
		\label{cor_reduce}
\end{cor}

\begin{proof}
		Consider the orbifold $S^2(2,2,2,n)$ in Figure \ref{fig_222n}. $O$ is the order-$n$ singular point, $C$, $D$ and $E$ are order-two singular points. $D$ and $E$ will be lifted to regular points in $S^2(2,2,n,n)$. 

		\begin{figure}[htbp]
				\centering
				\includegraphics{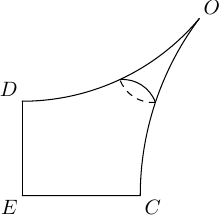}
				\caption{}
				\label{fig_222n}
		\end{figure}
		First we prove that curves in $\tilde{l}_D$ and $\tilde{l}_E$ cannot be lifted to be a systole on $\Gamma(2,n)$ surface. $\forall l \in \tilde{l}_D$, $l$ is lifted to a curve with self-intersections (See Figure \ref{fig_self_int}). Therefore $l$ cannot be lifted to a systole by Lemma \ref{lem_hyp_ell_n}. 

\begin{figure}[htbp]
		\centering
		\includegraphics{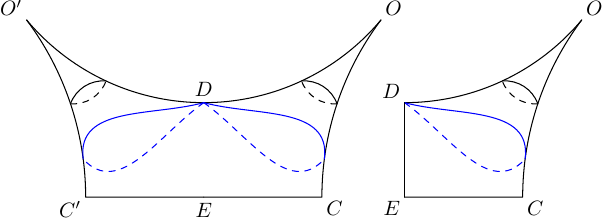}
		\caption{}
		\label{fig_self_int}
\end{figure}

Now we prepare a tool for the following proof. 
We construct the cyclic cover of the orbifold $S^2(2,2,2,n)$, as the construction of cyclic cover of knot completement constructed by Seifert surface. 

We first construct the $S^2(2,2,n,n)$ from $S^2(2,2,2,n)$. Two order-two singular points of $S^(2,2,2,n)$ are lifted to regular points of $S^2(2,2,n,n)$ (points $D$ and $E$ in Figure \ref{fig_222n}). We pick a curve $l$ connecting $D$ and $E$ (in other words $l\in \tilde{l}_{DE}$). We cut the regular neighbourhood ${N(\mathring{l})}$ of the interior of $l$ away, obtaining an orbifold with boundary (See Figure \ref{fig_cyclic_cover}). The boudary of $S^(2,2,2,n)\backslash N(\mathring{l})$ are divided into two parts by $D$ and $E$, we call them $l^+$ and $l^-$ respectively (See Figure \ref{fig_cyclic_cover_5}). 

\begin{figure}[htbp]
		\centering
		\subfigure[]{
		\includegraphics{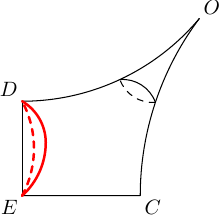}
		\label{fig_cyclic_cover_4}
}
		\subfigure[]{
		\includegraphics{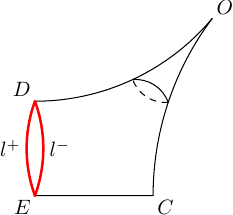}
		\label{fig_cyclic_cover_5}
}
		\caption{}
		\label{fig_cyclic_cover}
\end{figure}

Then we pick two copies of $S^(2,2,2,n)\backslash N(\mathring{l})$ calling them $P_1$ and $P_2$ respectively. We attach $P_1$'s $l^+$ to $P_2$'s $l^-$;  $P_1$'s $l^-$ to $P_2$'s $l^+$. Therefore we get the orbifold $S^2(2,2,n,n)$, the double cover of $S^2(2,2,2,n)$. (Figure \ref{fig_cyclic_cover_double}). 

\begin{figure}[htbp]
		\centering
		\subfigure[]{
		\includegraphics{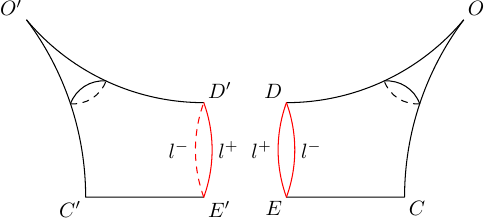}
		\label{fig_cyclic_cover_2}
}
		\subfigure[]{
		\includegraphics{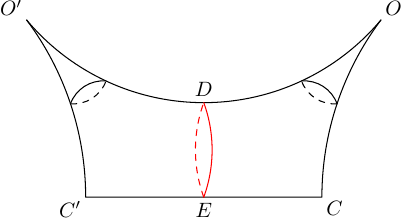}
		\label{fig_cyclic_cover_3}
}
		\caption{}
		\label{fig_cyclic_cover_double}
\end{figure}

Similarly, we construct $S^2(2,2,\dots,2)$, the order-$n$ cyclic cover of $S^2(2,2,n,n)$ from $S^2(2,2,n,n)$. 

The two order-$n$ singular points of $S^2(2,2,n,n)$ (the points $O$ and $O'$ in Figure \ref{fig_cyclic_cover_3}) will be lifted to regular points in $S^2(2,2,\dots,2)$. We pick a segment $l$ connecting $O$ and $O'$ (in other words $l\in \tilde{l}_{OO'}$) and cut away $N(\mathring{l})$ (Figure \ref{fig_cyclic_cover_n}). 

\begin{figure}[htbp]
		\centering
		\subfigure[]{
		\includegraphics{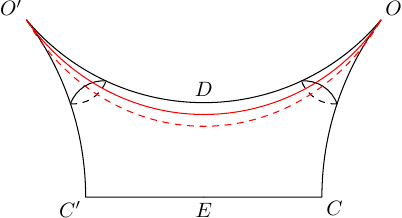}
		\label{fig_cyclic_cover_n_1}
}
		\subfigure[]{
		\includegraphics{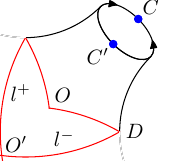}
		\label{fig_cyclic_cover_n_2}
}
		\caption{}
		\label{fig_cyclic_cover_n}
\end{figure}

The boundary of $S^2(2,2,n,n)\backslash N(\mathring{l})$ is divided into two pieces by $O$ and $O'$. We call the two pieces $l^+$ and $l^-$ respectively. We pick $n$ copies of $S^2(2,2,n,n)\backslash N(\mathring{l})$, denoted $P_1$, $P_2$, \dots, $P_n$ respectively. We attach $P_i$'s $l^+$ to $P_{i+1}$'s $l^-$ (assign $P_{n+1}=P_1$). Then we get $S^2(2,2,\dots,2)$ the $n$-cyclic cover of $S^2(2,2,2n,n)$ (Figure \ref{fig_cyclic_cover_n_n}). 

\begin{figure}[htbp]
		\centering
		\includegraphics{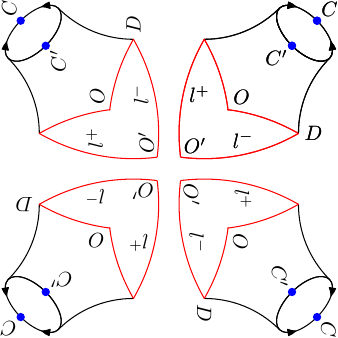}
		\includegraphics{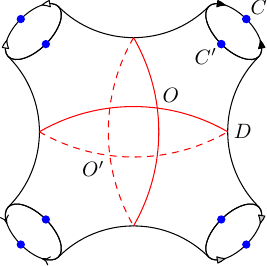}
		\caption{}
		\label{fig_cyclic_cover_n_n}
\end{figure}

Then we construct the surface $\Sigma$, the double cover of the orbifold $S^2(2,2,\dots,2)$. We construct it from $S^2(2,2,n,n)$. In $S^2(2,2,n,n)$ (Figure \ref{fig_cyclic_cover_3}), we pick a curve $l$ connecting $C$ and $C'$ (in other words $l\in \tilde{l}_{CC'}$). We cut the regular neighbourhood of $\mathring{l}$ and get $S^2(2,2,n,n)\backslash N(\mathring{l})$. The boundry of $S^2(2,2,n,n)\backslash N(\mathring{l})$ are divided into two pieces by $C$ and $C'$. We call the two pieces $l^-$ and $l^+$ respectively. Then we consider the pre-image $\pi'^{-1}(S^2(2,2,n,n)\backslash N(\mathring{l}))\subset S^2(2,2,2\dots,2)$. (Figure \ref{fig_cyclic_cover_n_5}) We pick two copies $P_1$ and $P_2$ of $\pi'^{-1}(S^2(2,2,n,n)\backslash N(\mathring{l}))$ then attach $P_1$'s $l^+$ to $P_2$'s $l^-$ and attach $P_1$'s $l^-$ to $P_2$'s $l^+$. Then we get the surface $\Sigma$. 

\begin{figure}[htbp]
		\centering
		\includegraphics{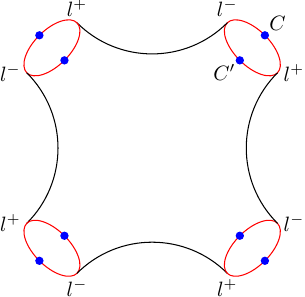}
		\caption{}
		\label{fig_cyclic_cover_n_5}
\end{figure}

The tool is now prepared. Now we continue to prove this theorem. 

(1) First we consider the family $\tilde{l}_C$. 
In the orbifold $S^2(2,2,2,n)$ (Figure \ref{fig_222n}), $\forall l \in \tilde{l}_C$, by Lemma \ref{lem_corresp_scc}, there exists $ l'\in \tilde{l}_{DE}$ such that $l\cap l' = \emptyset$. (Figure \ref{fig_22nnscc_1})

Then by the construction of $S^2(2,2,2,n)$'s double cover using $l'\in \tilde{l}_{DE}$, $l$ is lifted to two copies with the length equal to $l$'s length. Therefore, $\forall l\in \tilde{l}_C$, 
\[
		|l(S^2(2,2,2,n))| = |l(S^2(2,2,n,n))|. 
\]

For any $ l\in \tilde{l}_C$ in $S^2(2,2,2,n)$, one side of $l$ contains the order-$n$ singular point $O$, the other side contains two order-two singular points $D$ and $E$. Therefore, we consider $l$'s lift in $S^2(2,2,n,n)$. With a little abuse of symbol, we denote the lift as $l$ (Figure \ref{fig_22nnscc_1}). Then for the two order $n$ singular points $O$ and $O'$, $l$ seperates $O$ and $O'$, 

\begin{figure}[htbp]
		\centering
		\subfigure[]{
		\includegraphics{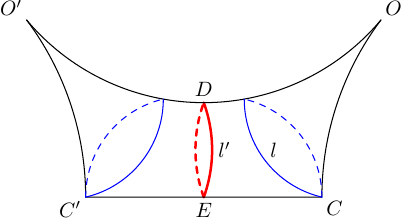}
		\label{fig_22nnscc_2}
}
		\subfigure[]{
		\includegraphics{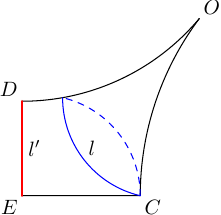}
		\label{fig_22nnscc_1}
}
		\caption{}
		\label{fig_22nnscc}
\end{figure}

We pick a curve $l''$ connecting $O$ and $O'$, $l''$ intersects $l$ exactly once. Then $l$ is cut into two pieces (denoted $l_1$ and $l_2$ respectively) by $l''$. Then in the construction of $S^2(2,2,\dots,2)$ by $l''$, for the $n$ copies of $S^2(2,2,n,n)\backslash N(\mathring{l''})$ (denoted $P_1, P_2, \dots, P_n$), $l_1$ in $P_1$ is attached to $l_2$ in $P_2$ (Figure \ref{fig_n_cover}). Therefore $l$ is lifted to a curve with length equal to $l$'s length. Therefore $\forall l\in \tilde{l}_C$, 
\[
		|l(S^2(2,2,2,n))| = |l(S^2(2,2,n,n))| = |l(S^2(2,2,\dots,2))|.
\]

.
\begin{figure}[htbp]
		\centering
		\includegraphics{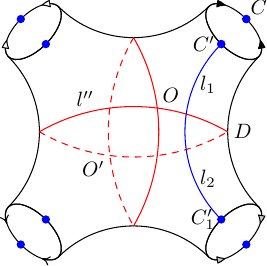}
		\includegraphics{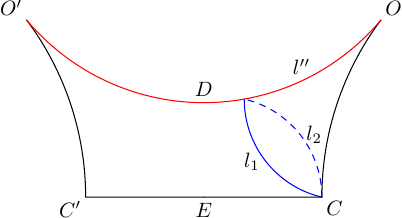}
		\caption{}
		\label{fig_n_cover}
\end{figure}

The end points of $l$ in $S^2(2,2,\dots,2)$ are $C'$ and $C_1'$ respectively. Therefore by the consturction of the surface $\Sigma$, the lift of $l$ in the surface $\Sigma$ consists of two pieces divided by the points $C$ and $C''$. Each piece has the length $|l(S^2(2,2,\dots,2))|$. Therefore, $\forall l\in \tilde{l}_C$, 
\[
		 |l(\Sigma)| = 2|l(S^2(2,2,\dots,2))|.
\]

(2) Then we consider the family $\tilde{l}_{CE}$. The proof here is similar to the proof for the family $\tilde{l}_{C}$. 

\begin{figure}[htbp]
		\centering
		\subfigure[]{
		\includegraphics{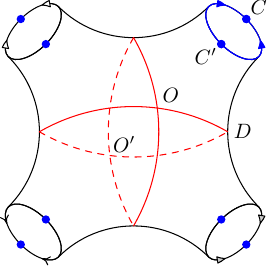}
		\label{fig_n_cover_ce_1}
}
		\subfigure[]{
		\includegraphics{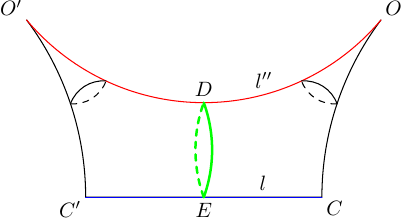}
		\label{fig_n_cover_ce_2}
}
		\subfigure[]{
		\includegraphics{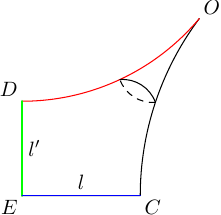}
		\label{fig_n_cover_ce_3}
}
		\caption{}
		\label{fig_n_cover_ce}
\end{figure}

For any $ l \in \tilde{l}_{CE}$, $l$ is a non-seperating curve in $S^2(2,2,2,n)$. Then there is an $l'\in \tilde{l}_{DE}$ such that $l'$ intersects $l$ only once at $E$ (Figure \ref{fig_n_cover_ce_3}). Then we construct $S^2(2,2,2,n)$'s double cover $S^2(2,2,n,n)$ by $l'$. The lift of $l$ in $S^2(2,2,n,n)$ is the blue curve in Figure \ref{fig_n_cover_ce_2}. It consists of two pieces, $CE$ and $C'E$. Each piece has the length equal to $l$'s length. Therefore $\forall l \in \tilde{l}_{CE}$
\[
		|l(S^2(2,2,n,n))| = 2 |l(S^2(2,2,2,n))|. 
\]

Then we lift $l$ to $S^2(2,2,\dots,2)$. $\forall l \in \tilde{l}_{CC'}$ in $S^2(2,2,n,n)$, $l$ is a non-seperating curve. Thus there exists $l''$ connecting $O$ and $O'$ such that $l''$ does not intersect $l$ (Figure \ref{fig_n_cover_ce_2}). Then we use $l''$ to construct $S^2(2,2,\dots,2)$, $S^2(2,2,n,n)$. Thus we know $l$ is lifted to $n$ disjoint segments in $S^2(2,2,\dots,2)$ with length equal to $l$'s. That is to say $\forall l \in \tilde{l}_{CE}$
\[
		|l(S^2(2,2,\dots,2))|=|l(S^2(2,2,n,n))|. 
\]

The last thing is to lift $l$ to the surface $\Sigma$. By exactly the same proof in (1), we have
\[
		|l(\Sigma)| = 2 |l(S^2(2,2,\dots,2))|.
\]

(3) Next we consider the family $\tilde{l}_{CD}$. $\forall l \in \tilde{l}_{CD}$, $l$ is lifted to a curve connecting $C$ and $C'$ in $S^2(2,2,n,n)$ (in other words, in the family $\tilde{l}_{CC'}$) (See Figure \ref{fig_cover_cd}). Then what we need to prove has already proved in (2). 

\begin{figure}[htbp]
		\centering
		\subfigure[]{
		\includegraphics{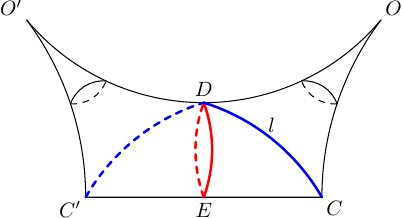}
		\label{fig_cover_cd_2}
}
		\subfigure[]{
		\includegraphics{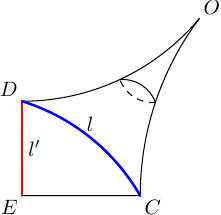}
		\label{fig_cover_cd_3}
}
		\caption{}
		\label{fig_cover_cd}
\end{figure}

(4) Finally we consider the family $\tilde{l}_{DE}$. $\forall l \in \tilde{l}_{DE}$, we use $l$ to construct $S^2(2,2,n,n)$, the double cover of $S^2(2,2,n,n)$. Then the lift of $l$ in $S^2(2,2,n,n)$ is a simple closed curve. This curve is divided into two pieces. Each piece has length equal to the length of $l$ (Figure \ref{fig_n_cover_de_2}, \ref{fig_n_cover_de_3}). Thus $\forall l \in \tilde{l}_{DE}$, 
\[
		|l(S^2(2,2,n,n))| = 2 |l(S^2(2,2,2,n))| . 
\]

\begin{figure}[htbp]
		\centering
		\subfigure[]{
		\includegraphics{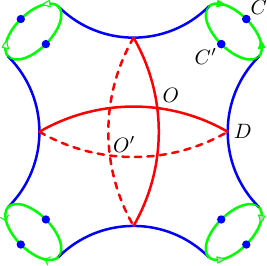}
		\label{fig_n_cover_de_1}
}
		\subfigure[]{
		\includegraphics{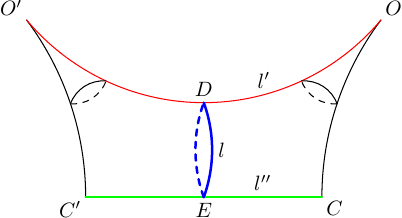}
		\label{fig_n_cover_de_2}
}
		\subfigure[]{
		\includegraphics{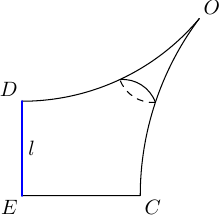}
		\label{fig_n_cover_de_3}
}
		\caption{}
		\label{fig_n_cover_de}
\end{figure}

Then we lift $l$ to $S^2(2,2,\dots,2)$ similar to the proof in (1), $l$ in $S^2(2,2,n,n)$ seperates $O$ and $O'$. We pick a segment $l'$ connecting $O$ and $O'$ such that $l'$ intersects $l$ only once. We use $l'$ to construct $S^2(2,2,\dots,2)$, the $n$-cyclic cover of $S^2(2,2,n,n)$. Since $l\backslash l'$ is a segment, $l$'s lift in $S^2(2,2,\dots,2)$ is a simple closed curve whose length is $n$-th the length of $l$. That is $\forall l\in \tilde{l}_{DE}$, 
\[
		|l(S^2(2,2,\dots,2))|=n|l(S^2(2,2,n,n))|. 
\]

At last we lift $l$ to the surface $\Sigma$, 
We use $l''\in \tilde{l}_{CC'}$ in Figure \ref{fig_n_cover_de_2} (and its pre-image in Figure \ref{fig_n_cover_de_1}) to construct $\Sigma$. $\pi'^{-1}(S^2(2,2,n,n)\backslash N(\mathring{l''}))$ is an $n$-holed sphere. We pick two copies of the $n$-holed spheres, denoted $P_1$ and $P_2$. We attach $P_1$ to $P_2$ along their boundaries (See Figure \ref{fig_n_cover_de_4}) In Figure \ref{fig_n_cover_de_4}, by the definition of $t$, $F_i$ in $P_1$ is attached to $G_{i+1}$ in $P_2$ and $G_i$ in $P_1$ is attached to $F_{i-1}$ in $P_2$. Here $F_i$ and $G_i$ are end points of components of $l$ in the $n$-holed spheres. Therefore, the lift of $l$ in $\Sigma$ consists of simple closed curve(s). Each curve is connected segments in $P_1$ or $P_2$. If $n$ is odd the lift of $l$ is one simple closed curve consists of the segments $G_1F_1^{(1)}$, $G_2F_2^{(2)}$, \dots, $G_nF_n^{(2)}$. If $n$ is even the lift of $l$ consists of two simple closed curves. One curve consists of the segments $G_1F_1^{(1)}$, $G_2F_2^{(2)}$, \dots, $G_nF_n^{(2)}$. The other curve consists of the segments $G_1F_1^{(2)}$, $G_2F_2^{(1)}$, \dots, $G_nF_n^{(1)}$.

\begin{figure}[htbp]
		\centering
		\subfigure[$P_1$]{
		\includegraphics{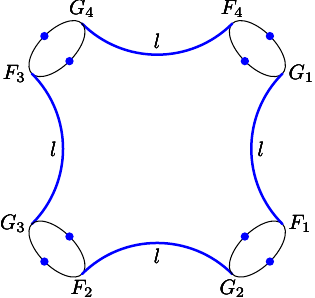}
}
		\subfigure[$P_2$]{
		\includegraphics{n_cover_3-4.pdf}
}
		\caption{}
		\label{fig_n_cover_de_4}
\end{figure}

Therefore, $\forall l \in \tilde{l}_{DE}$, 

\[
		|l(\Sigma)| = 2|l(S^2(2,2,\dots,2))|
\]
if $n$ is odd; while 
\[
		|l(\Sigma)| = |l(S^2(2,2,\dots,2))|
\]
if $n$ is even. 
\end{proof}

Now we have proved the theorem. We list the length ratio we obtained in  the following Table: 
\begin{table}[h]
		\centering
		\begin{tabular}{lllll}
				& $\tilde{l}_{C}$ & $\tilde{l}_{CD}$ & $\tilde{l}_{CE}$ & $\tilde{l}_{DE}$  \\
				$\dfrac{|l(S^2(2,2,n,n))|}{|l(S^2(2,2,2,n))|}$ & 1 & 2 & 2 & 2 \\
				$\dfrac{|l(S^2(2,2,\dots,2))|}{|l(S^2(2,2,n,n))|}$ & 1 & 1 & 1 & n \\ 
				$\dfrac{|l(\Sigma)|}{|l(S^2(2,2,\dots,2))|}$ & 2 & 2 & 2 & 2 if n is odd; 1 if n is even 
		\end{tabular}
		\caption{}
		\label{tab_ratio}
\end{table}

\section{Shortest curve in each family}
\label{sec_shortest}

By Lemma \ref{lem_hyp_ell_n_2}, there are only four families of curves in $S^2(2,2,2,n)$ ($\tilde{l}_C$, $\tilde{l}_{CE}$, $\tilde{l}_{CD}$ and $\tilde{l}_{DE}$) that are possible to be lifted to a systole. 
By Corollary \ref{cor_reduce}, in each family, only the shortest curve is possible to be lifted to a systole in the surface. 

Then we characterize the shortest curve in each family in $S^2(2,2,2,n)$ by the Proposition \ref{prop_reduce}. 

Before stating Proposition \ref{prop_reduce}, we give a short preparation: 

We recall that the pentagon in Figure \ref{fig_orbi_geo_3} is a model to describe the geometry of the orbifold $S^2(2,2,2,n)$. 

The vertices of the pentagon in Figure \ref{fig_orbi_geo_3} that corresponds to the same point of $S^2(2,2,2,n)$ are labeled by the same letter. To avoid ambiguity, we replace one $A$ and one $D$ by $A_1$ and $D_1$ respectively (Figure \ref{fig_polygon_3}). 

A curve in the orbifold $S^2(2,2,2,n)$ corresponds to broken segments in the pentagon. (See for example Figure \ref{fig_00_curve_1}) The pentagon is symmetric. There is a 'reflection' (orietation-reversing, isometric map) of the pentagon mapping $A_1$ to $A$ and $D_1$ to $D$. We reflect some components of the segments corresponding to the curve, then get a connected broken line with the same length to the segments (the dashed blue line in Figure \ref{fig_00_curve_2}). By this construction, one of the endpoint of the broken line is an endpoint of the segments (point $C$ in Figure \ref{fig_00_curve}); while the other endpoint of the broken line is either the other endpoint of the segments (point $D$ in Figure \ref{fig_00_curve}) or the reflection of the endpoint of the segments (point $D_1$ of Figure \ref{fig_00_curve}). 

\begin{figure}[htbp]
		\centering
		\includegraphics{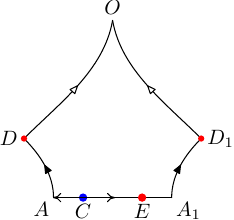}
		\caption{}
		\label{fig_polygon_3}
\end{figure}

\begin{proposition}
		The shortest curve among a family of curves ($\tilde{l}_C$, $\tilde{l}_{CE}$, $\tilde{l}_{DE}$ or $\tilde{l}_{CD}$) in $S^2(2,2,2,n)$ corresponds to the broken segments with the least number of components. 

		Moreover, the shortest curve in each family is shown in Figure \ref{fig_sys_cand}. 
		\label{prop_reduce}
\end{proposition}

By this Proposition, the possible systoles of the $\Gamma(2,n)$ surface are reduced to finitely many curves. 

\begin{proof}
		The proof for all these cases are similar although they are different in details.

		(1) For $\tilde{l}_{CD}$, the shortest curve in this family is the segment $CD$ in the pentagon $ODAA_1D_1$ (Figure \ref{fig_sys_00}). By the symmetry of the pentagon, $\forall l \in \tilde{l}_{CD}$, there is a connected broken line connecting $CD$ or $CD_1$, with the length of $l$ (The dashed blue line in Figure \ref{fig_00_curve_2}). The broken line connects $CD$ or $CD_1$, therefore is longer than the straight line connecting $CD$ or $CD_1$. But $CD$ is always shorter than $CD_1$, because in the right-angled triangle $\triangle CAD$ and $\triangle CA_1D_1$, $\angle A = \angle A_1 = \pi/2$, $AD = A_1D_1$ while $AC = t/2 < c-t/2 = A_1C_1$ (see Section \ref{sec_geo_orbi}). Therefore the shortest curve in $\tilde{l}_{CD}$ is the segment connecting $C$ and $D$ in the pentagon, denoted $l_{CD}$. (See Figure \ref{fig_sys_00}).

\begin{figure}[htbp]
		\centering
		\subfigure[]{
		\includegraphics{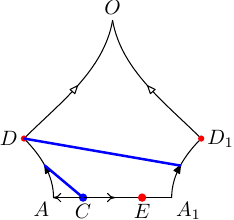}
		\label{fig_00_curve_1}
}
		\subfigure[]{
		\includegraphics{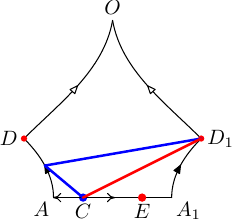}
		\label{fig_00_curve_2}
}
		\caption{}
		\label{fig_00_curve}
\end{figure}

(2) For $\tilde{l}_{DE}$, the proof is exactly the same by the symmetry of the pentagon. The shortest curve in $\tilde{l}_{DE}$ is the straight segment connecting $D_1$ and $E$, denoted $l_{DE}$ (see Figure \ref{fig_sys_11}). 

(3) For  $\tilde{l}_{CE}$, the proof is similar. 

(3.1) For $l\in \tilde{l}_{CD}$, if $l$ ($l$ consists of segments in the pentagon) has a component connecting $AD$, $OD$ or $A_1D_1$, $OD_1$ (see Figure \ref{fig_cuff_1}), then we reflect this segment and the image of the segment connects two other segments of $l$ (Figure \ref{fig_cuff_2}). Therefore we get a curve, shorter than $l$, and has less intersections with $OD$ than $l$ (Figure \ref{fig_cuff_4}). 

\begin{figure}[htbp]
		\centering
		\subfigure[]{
		\includegraphics{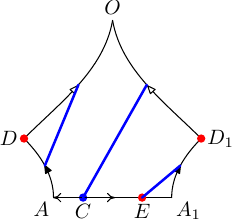}
		\label{fig_cuff_1}
}
		\subfigure[]{
		\includegraphics{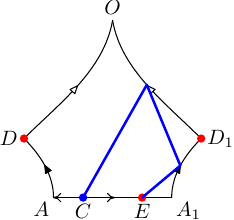}
		\label{fig_cuff_2}
}
		\subfigure[]{
		\includegraphics{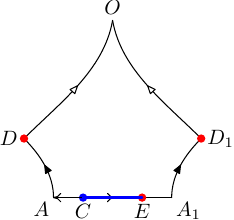}
		\label{fig_cuff_4}
}
		\caption{}
		\label{fig_cuff}
\end{figure}

(3.2) If all of the components of $l$ are segments connecting $AA_1$, $OD$ or $AA_1$, $OD_1$, and there exist components not meeting $C$ or $E$, then we pick two such segments, $A_2D_2$ and $A_3D_3$ in Figure \ref{fig_cuff_5}, $A_2$ and $A_3$ correspond to the same point in the orbifold. Then we replace $A_2D_2$ by $A_3D_5$ (here $A_2$ and $A_3$ correspond to the same point in the orbifold; $D_2$ and $D_5$ correspond to the same point in the orbifold) (see Figure \ref{fig_cuff_6}); or we replace  $A_3D_3$ by $A_2D_4$ (here $A_2$ and $A_3$ correspond to the same point in the orbifold; $D_3$ and $D_4$ correspond to the same point in the orbifold) (see Figure \ref{fig_cuff_7}). 

\begin{figure}[htbp]
		\centering
		\subfigure[]{
		\includegraphics{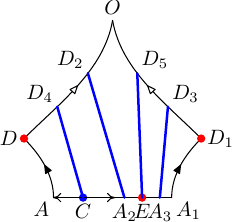}
		\label{fig_cuff_5}
}
		\subfigure[]{
		\includegraphics{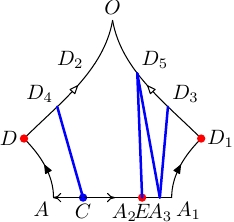}
		\label{fig_cuff_6}
}
		\subfigure[]{
		\includegraphics{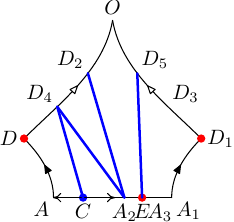}
		\label{fig_cuff_7}
}
		\subfigure[]{
		\includegraphics{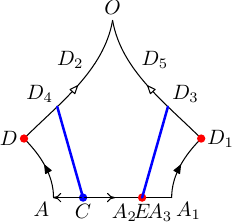}
		\label{fig_cuff_8}
}
		\caption{}
		\label{fig_cuff_reduce_2}
\end{figure}

Segments in Figure \ref{fig_cuff_6} and Figure \ref{fig_cuff_7} correspond to a curve in the family $\tilde{l}_{CE}$. One of the two curves is shorter than $l$, the curve in Figure \ref{fig_cuff_5} (explained later). Then, without loss of generality, we assume the curve in Figure \ref{fig_cuff_6} is shorter. By replacing the broken line $ED_5A_3D_3$ in Figure \ref{fig_cuff_6} by the straight line $ED_3$, we get a curve (Figure \ref{fig_cuff_8}) homotopic to the curve in Figure \ref{fig_cuff_6}. This curve is shorter than $l$ and has less intersections with $OD$ than $l$. 

(3.3) $\forall l \in \tilde{l}_{CE}$ we use the operations described in (3.1) and (3.2) to change $l$ until we can not use the operations. Everytime we use the operations, we get a curve shorter and has less intersections with $OD$ than the original curve. Finally we get the curve in Figure \ref{fig_sys_cuff} (dnonted $l_{CE}$) or the curve in Figure \ref{fig_sys_cuff_fake} (denoted $l'_{CE}$). We prove later in Lemma xxx that $l'_{CE}$ cannot be lifted to a systole. 

(3.4) One thing left to prove: in the operation described in (3.2), the curve in Figure \ref{fig_cuff_5} is longer than the curve in Figure \ref{fig_cuff_6} or the curve in Figure \ref{fig_cuff_7}. We prove it with the following fomular \cite[(3.10)]{bgw}: 
\begin{eqnarray}
		\cosh c &=& \cosh d \cosh a \cosh b-\sinh a \sinh b \label{for_birect_1}. 
		\label{for_quad}
\end{eqnarray}

\begin{figure}[htbp]
\centering
\includegraphics{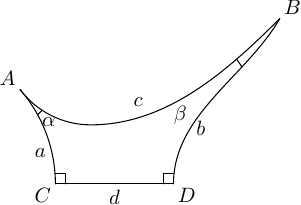}
\caption{}
\label{fig_2_right_angle_quadrilateral}
\end{figure}

The meaning of the symbols in the formula is illustrated by Figure \ref{fig_2_right_angle_quadrilateral}. In Figure \ref{fig_cuff_5}, 
\[
		\cosh |A_2D_2| = \cosh |AD| \cosh |AA_2| \cosh |DD_2| - \sinh |AA_2| \sinh |DD_2|. 
\]

While in Figure \ref{fig_cuff_6}, 
\[
		\cosh |A_3D_5| = \cosh |A_1D_1| \cosh |A_1A_3| \cosh |D_1D_5| - \sinh |A_1A_3| \sinh |D_1D_5|. 
\]
Here $|AD| = |A_1D_1|$ by the symmetry of the pentagon. $|DD_2| = |D_1D_5|$ because $|D_2|$ and $D_5$ correspond to the same point in the orbifold and so are $D$ and $D_1$. 

Therefore curve in Figure \ref{fig_cuff_5} is longer than curve in Figure \ref{fig_cuff_6} if and only if $|A_2D_2| > |A_3D_5|$. $|A_2D_2| > |A_3D_5|$ if and only if $|AA_2| > |A_1A_3|$. 

Similarly, in Figure \ref{fig_cuff_5}, 
\[
		\cosh |A_3D_3| = \cosh |A_1D_1| \cosh |A_1A_3| \cosh |D_1D_3| - \sinh |A_1A_3| \sinh |D_1D_3|. 
\]

While in Figure \ref{fig_cuff_7}, 
\[
		\cosh |A_2D_4| = \cosh |AD| \cosh |AA_2| \cosh |DD_4| - \sinh |AA_2| \sinh |DD_4|. 
\]
Here $|AD| = |A_1D_1|$ by the symmetry of the pentagon. $|DD_4| = |D_1D_3|$ because $|D_4|$ and $D_3$ correspond to the same point in the orbifold and so are $D$ and $D_1$. 

Therefore curve in Figure \ref{fig_cuff_5} is longer than curve in Figure \ref{fig_cuff_7} if and only if $|A_3D_3| > |A_2D_4|$. $|A_3D_3| > |A_2D_4|$ if and only if $|AA_2| < |A_1A_3|$. 

In conclusion, the curve in Figure \ref{fig_cuff_5} is longer than either the curve in Figure \ref{fig_cuff_6} or the curve in Figure \ref{fig_cuff_7}. 

(4) For $\tilde{l}_C$, the proof is similar to (1) and (2). $\forall l \in \tilde{l_C}$ (Figure \ref{fig_01_curve_2}), by using the reflection, we can construct two connecting broken lines. One line connects $C$ and a point on $OD$ ($D_6$ in Figure \ref{fig_01_curve_3}), while the other connects $C$ and a point on $OD_1$ ($D_7$ in Figure \ref{fig_01_curve_3}). $D_6$ and $D_7$ correspond to the same point in the orbifold. The length of the broken lines are equal to $l$ since all the changes are reflections. 

\begin{figure}[htbp]
		\centering
		\subfigure[]{
		\includegraphics{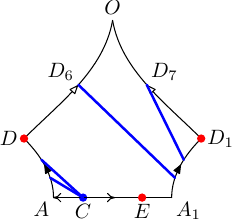}
		\label{fig_01_curve_2}
}
		\subfigure[]{
		\includegraphics{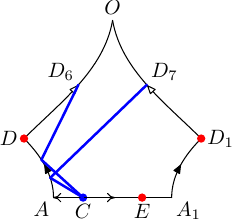}
		\label{fig_01_curve_3}
}
		\subfigure[]{
		\includegraphics{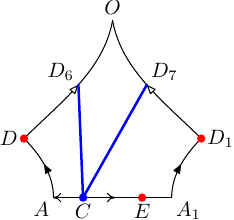}
		\label{fig_01_curve_1}
}
		\caption{}
		\label{fig_01_curve}
\end{figure}

Then the straight lines connecting $C,D_6$ and $C,D_7$ are shorter than the corresponding broken lines in Figure \ref{fig_01_curve_3}. Therefore, the shortest curve in the family $\tilde{l}_C$ is the closed geodesic in the orbifold consisting of two straight lines in the pentagon, one connecting $C$ and a point in $OD$, the other connecting $C$ and a point in $OD_1$ (denoted $l_C$). (See Figure \ref{fig_sys_01})

\begin{figure}[htbp]
		\centering

		\subfigure[Curve $l_{CD}$]{
		\includegraphics{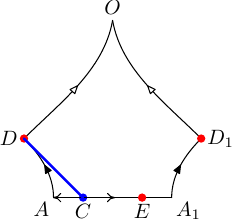}
		\label{fig_sys_00}
}
\subfigure[Curve $l_{DE}$]{
		\includegraphics{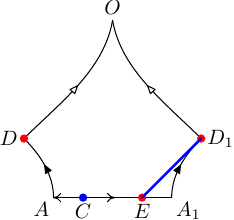}
		\label{fig_sys_11}
}
\subfigure[Curve $l_{CE}$]{
		\includegraphics{cuff-4.pdf}
		\label{fig_sys_cuff}
}
\subfigure[Curve $l'_{CE}$]{
		\includegraphics{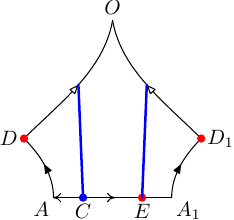}
		\label{fig_sys_cuff_fake}
}
		\subfigure[Curve $l_C$]{
		\includegraphics{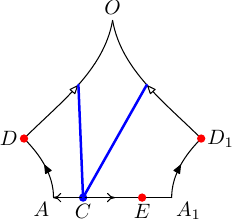}
		\label{fig_sys_01}
}

		\caption{}
		\label{fig_sys_cand}
\end{figure}

\end{proof}

\section{Calculations}

\subsection{The curve $l_{DE}$ and $l'_{CE}$}

In this subsection, we prove that the curve $l_{DE}$ and $l'_{CE}$ cannot be lifted to systoles.

First we give the following Proposition: 

\begin{proposition}
		The curve $l'_{CE}$ in $S^2(2,2,2,n)$ is not possible to lift to a systole in the $\Gamma(2,n)$ surface $\Sigma$. 
		\label{prop_cuff_fake_not}
\end{proposition}
\begin{proof}
		We prove this Proposition by cut and paste of the pentagon. In the pentagon (Figure \ref{fig_cut_paste_1}), we cut along $CD$ and $D_1E$ then paste along $CE$. Then we get a quadrilateral (Figure \ref{fig_cut_paste_2}). 
\begin{figure}[htbp]
		\centering
		\subfigure[]{
		\includegraphics{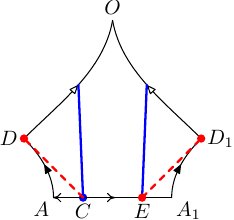}
		\label{fig_cut_paste_1}
}
		\subfigure[]{
		\includegraphics{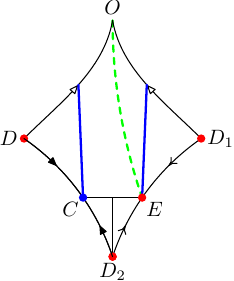}
		\label{fig_cut_paste_2}
}
		\subfigure[]{
		\includegraphics{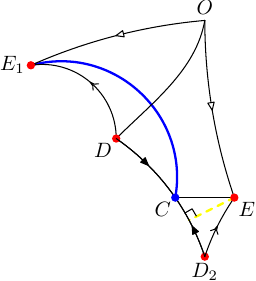}
		\label{fig_cut_paste_3}
}
		\subfigure[]{
		\includegraphics{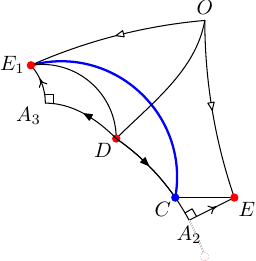}
		\label{fig_cut_paste_4}
}
		\caption{}
		\label{fig_cut_paste}
\end{figure}

In The quadrilateral in Figure \ref{fig_cut_paste_2}, we cut along the dashed green segment $OE$, then paste along the segments $OD$ and $OD_1$. Then we get a pentagon (Figure \ref{fig_cut_paste_3}). 

In The pentagon in Figure \ref{fig_cut_paste_3}, we cut along the dashed yellow segment that is from $E$ and perpendicular to $DD_2$, then paste along the segments $DE_1$ and $D_1D$. We denote the foot of the perpendicular to be $A_2$. Then we get a pentagon (Figure \ref{fig_cut_paste_4}). The pentagon in Figure \ref{fig_cut_paste_4} has two right angles, $A_2$ and $A_3$. 

In all the four subfigures of Figure \ref{fig_cut_paste}, the blue segments represent the curve $l'_{CE}$, while the segment $CE$ always represents the curve $l_{CE}$. Then since in Figure \ref{fig_cut_paste_4}, $|A_2A_3| = 2|CD|$, therefore $|A_3C|> |A_2C|$. $|A_2E| = |A_3E_1|$. Then by hyperbolic cosine law, $CE < CE_1$. 

\end{proof}

Then we give the lengths of the curves in Figure \ref{fig_sys_cand} by $c,s$ and $t$. It is a preparation for proving Proposition \ref{prop_11_not}. We recall that in the pentagon $|CE|=c/2$, $|AC|=t/2$, $|EA_1| = (c-t)/2$ and $|AD| = |A_1D_1| = s/2$ (see Section \ref{sec_geo_orbi}). 

It is direct that, for the curve in Figure \ref{fig_sys_cuff}, 
\begin{equation}
		|l_{CE}| = \frac{c}{2}. 
		\label{for_sys_cuff}
\end{equation}
We calculate the lengths of $l_{CD}$ (Figure \ref{fig_sys_01}) and $l_{CD_1}$ (Figure \ref{fig_sys_11}) by the cosine law of hyperbolic right-angled triangles: 

\begin{align}
	\cosh |l_{CD}| &= \cosh |CD|	\label{for_sys_00} \\
\nonumber		&= \cosh |AC| \cosh |AD| \\
\nonumber		&= \cosh \frac{t}{2} \cosh \frac{s}{2}. 
\end{align}

\begin{align}
		\cosh |l_{DE}| &= \cosh |D_1E| \label{for_sys_11}\\
		&= \cosh |EA_1| \cosh |A_1D_1| \nonumber\\
		&= \cosh \frac{c-t}{2} \cosh \frac{s}{2}. \nonumber
\end{align}

For $l_{C}$ and $l'_{CE}$, to calculate their lengths, we attach a copy of the pentagon to its edge $OD$ (Figure \ref{fig_11_01_two_1}). The length of the segmenmet $CE'$ is equal to the length of the curve $l'_{CE}$ in Figure \ref{fig_sys_11} and the length of the senge $CC'$ is equal to the length of the curve $l_C$ in Figure \ref{fig_sys_01} by symmetry. We use Formula (\ref{for_birect_1}) to calculate these lengths. Here $|AA_1'| = 2|AD| = s$, $|AC| = t$, $|A_1'E'| = |A_1E| = (c-t)/2$, and $|A_1'C' | = |A_1C| = c - t/2$. Then

\begin{align}
		\cosh l'_{CE} &= \cosh |CE'| \label{for_sys_cuff_fake} \\
		\nonumber &= \cosh |AA_1'| \cosh |AC| \cosh |A_1'E'| - \sinh |AC| \sinh |A_1'E'| \\
		\nonumber &= \cosh s \cosh \frac{t}{2} \cosh \frac{c-t}{2} - \sinh \frac{t}{2} \sinh \frac{c-t}{2}. 
\end{align}

\begin{figure}[htbp]
		\centering
		\includegraphics{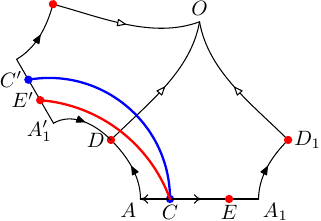}
		\caption{}
		\label{fig_11_01_two_1}
\end{figure}

\begin{align}
		\cosh l_{C} &= \cosh |CC'| \label{for_sys_01} \\
		\nonumber &= \cosh |AA_1'| \cosh |AC| \cosh |A_1'C'| - \sinh |AC| \sinh |A_1'C'| \\
		\nonumber &= \cosh s \cosh \frac{t}{2} \cosh c-\frac{t}{2} - \sinh \frac{t}{2} \sinh c-\frac{t}{2}. 
\end{align}

Now we are ready to prove the following Proposition: 

\begin{proposition}
		In the maximal surface, $l_{DE}$ in $S^2(2,2,2,n)$ is not possible bo be lifted to a systole of the surface. 
		\label{prop_11_not}
\end{proposition}

\begin{proof}
		The curves in the orbifold $S^2(2,2,2,n)$ that are possible to be lifted to the systole of the $\Gamma(2,n)$ surface are $l_{CD}$, $l_{DE}$, $l_{CE}$ and $l_{C}$. 
		If $l_{DE}$ is lifted to a systole of the surface, then $l_{CD}$ and $l_{CE}$ cannot be lifted to a systole of this surface. This is because $l_{CD}$ and $l_{CE}$ intersect $l_{DE}$ at $D$ and $E$ respectively. $D$ and $E$ are lifted to regular points in $S^2(2,2,n,n)$. Then by Lemma \ref{lem_hyp_ell_n}, since $l_{DE}$ is lifted to a systole, $l_{CD}$ and $l_{CE}$ cannot be lifted to systoles. 

		If $l_{DE}$ is lifted to a systole of the surface, then only $l_{DE}$ and $l_{C}$ can be lifted to systoles of the surface. The lengths of $l_{DE}$ and $l_{C}$ are given by (\ref{for_sys_11}) and (\ref{for_sys_01}) respectively. Here we give the differentials of the lengths:

		First we obtain $\mathrm{d}s/\mathrm{d}c$ by (\ref{for_cs}):
		\begin{eqnarray*}
				\mathrm{d} (\sinh \frac{c}{2} \sinh \frac{s}{2}) &=& \mathrm{d} \cos \frac{\pi}{n} = 0\\
				\cosh \frac{s}{2} \sinh \frac{c}{2}\mathrm{d}s + \cosh \frac{c}{2} \sinh \frac{s}{2}\mathrm{d}c &=& 0 \\
				\frac{\mathrm{d}s}{\mathrm{d}c} &=& -\frac{\cosh \frac{c}{2} \sinh \frac{s}{2}}{\cosh \frac{s}{2} \sinh \frac{c}{2}}. 
		\end{eqnarray*}

		Then for $l_{DE}$
		\begin{eqnarray*}
				\frac{\partial |l_{DE}|}{\partial t} &=& \frac{\partial }{\partial t} \left( \cosh \frac{s}{2} \cosh \frac{c-t}{2} \right) \\
				&=& - \frac{1}{2} \cosh \frac{s}{2} \sinh \frac{c-t}{2}. \\
				\frac{\partial |l_{DE}|}{\partial c} &=& \frac{\partial }{\partial c} \left( \cosh \frac{s}{2} \cosh \frac{c-t}{2} \right) \\
				&=& \frac{1}{2} \left( \sinh \frac{s}{2} \frac{\mathrm{d}s}{\mathrm{d}c} \cosh \frac{c-t}{2} + \cosh \frac{s}{2} \sinh \frac{c-t}{2} \right) \\
				&=& \frac{1}{2} \left(  -\frac{\cosh \frac{c}{2} \sinh^2 \frac{s}{2}}{\cosh \frac{s}{2} \sinh \frac{c}{2}}\cosh \frac{c-t}{2} + \cosh \frac{s}{2} \sinh \frac{c-t}{2} \right). \\
				\mathrm{d}|l_{DE}| &=& \frac{\partial |l_{DE}|}{\partial t}\mathrm{d}t + \frac{\partial |l_{DE}|}{\partial c}\mathrm{d}c. 
		\end{eqnarray*}
		For $l_{C}$
		\begin{align}
				\label{for_sys_01_dt}\frac{\partial |l_{C}|}{\partial t} =& \frac{\partial }{\partial t} \left( \cosh s \cosh \frac{t}{2} \cosh \left( c-\frac{t}{2} \right) - \sinh \frac{t}{2} \sinh \left( c-\frac{t}{2} \right)  \right) \\
				\nonumber=& \frac{1}{2} \left( \cosh s \left( \sinh \frac{t}{2} \cosh \left( c-\frac{t}{2} \right)-  \cosh \frac{t}{2} \sinh \left( c-\frac{t}{2} \right)\right) \right.\\
				\nonumber &\left. -  \cosh \frac{t}{2} \sinh \left( c-\frac{t}{2} \right) + \sinh \frac{t}{2} \cosh \left( c-\frac{t}{2} \right)\right) \\
				\nonumber=& \frac{1}{2} \left( \cosh s +1 \right) \sinh\left( t-c \right). 
		\end{align}
		\begin{eqnarray*}
				\frac{\partial |l_{C}|}{\partial c} &=& \frac{\partial }{\partial c} \left( \cosh s \cosh \frac{t}{2} \cosh \left( c-\frac{t}{2} \right) - \sinh \frac{t}{2} \sinh \left( c-\frac{t}{2} \right)  \right) \\
				&=&  \sinh s \frac{\mathrm{d}s}{\mathrm{d}c} \cosh \frac{t}{2}\cosh \left( c-\frac{t}{2} \right) + \cosh s \cosh \frac{t}{2} \sinh\left( c-\frac{t}{2} \right) \\
				& &- \sinh \frac{t}{2} \cosh\left( c-\frac{t}{2} \right) \\
				&=& -\frac{\cosh \frac{c}{2} \sinh \frac{s}{2}}{\cosh \frac{s}{2} \sinh \frac{c}{2}}\sinh s  \cosh \frac{t}{2}\cosh \left( c-\frac{t}{2} \right) + \cosh s \cosh \frac{t}{2} \sinh\left( c-\frac{t}{2} \right) \\
				& &- \sinh \frac{t}{2} \cosh\left( c-\frac{t}{2} \right).\\
				\mathrm{d}|l_{C}| &=& \frac{\partial |l_{C}|}{\partial t}\mathrm{d}t + \frac{\partial |l_{C}|}{\partial c}\mathrm{d}c. 
		\end{eqnarray*}

		The two tangent vectors $\mathrm{d}|l_{DE}|$, $ \mathrm{d}|l_{C}|$ are non-zero vectors. 
		$\forall c>0, 0\le t\le c$, 
		\[
				\frac{\partial |l_{DE}|}{\partial t} <0, \frac{\partial |l_{C}|}{\partial t} <0. 
		\]
		Therefore $\mathrm{d}|l_{DE}| \ne k \mathrm{d}|l_{C}|$ $\forall k\le0$. 
		Then there is a vector $(A(c,t), B(c,t))$ such that 
\[
		\mathrm{d}|l_{DE}|(A(c,t), B(c,t))>0, \mathrm{d}|l_{C}|(A(c,t), B(c,t))>0, 
\]
$\forall c>0, 0\le t\le c$. By the assumption that $l_{DE}$ is lifted to a systole of the surface, only $l_{DE}$ and $l_{C}$ can be lifted to a systole of the surface. Then there is another surface with systole bigger than the surface. Therefore the surface is not maximal
\end{proof}

\subsection{$l_{CE}$, $l_{CD}$ and $l_{C}$ in the maximal surface}

Now only $l_{CE}$, $l_{CD}$ and $l_{C}$ can be lifted to a systole of the maximal surface. 

We have the following Proposition:
\begin{proposition}
		If $\Sigma_0$ is the maximal $\Gamma(2,n)$ surface, then 
		\[
				|l_{C}(\Sigma_0)| = |l_{CD}(\Sigma_0)| = |l_{CE}(\Sigma_0)|. 
		\]
		\label{prop_equal}
\end{proposition}
\begin{proof}
		First we calculate the partial dirivitives of the lengths. 	It is direct that 
		\begin{equation}
				\label{for_sys_cuff_dt}\frac{\partial |l_{CE}|}{\partial t} = \frac{\partial c}{\partial t} = 0. 
		\end{equation}
		We have obtained $\frac{\partial |l_{C}|}{\partial t}$ in (\ref{for_sys_01_dt}). For $\frac{\partial |l_{CD}|}{\partial t}$, we have the following formula: 
		\begin{align}
				\label{for_sys_00_dt}\frac{\partial |l_{CD}|}{\partial t} =& \frac{\partial }{\partial t}\left( \cosh \frac{t}{2}\cosh \frac{s}{2} \right)\\
				\nonumber =& \frac{1}{2} \sinh \frac{t}{2}\cosh \frac{s}{2} . 
		\end{align}

		For any fixed $ c>0$, $|l_{CD}|$ is strictly increasing about $t$, while $|l_{C}|$ is strictly decreasing about $t$ when $0\le t\le c$ by (\ref{for_sys_00_dt}) and (\ref{for_sys_01_dt}) respectively. By table \ref{tab_ratio}, $|l_{CD}(\Sigma)| = 4 |l_{CD}|$, while $|l_{C}(\Sigma)| = 2 |l_{C}|$. Then we compare $|l_{CD}(\Sigma)|$ with $|l_{C}(\Sigma)|$ by comparing $|l_{CD}|$ and $2|l_C|$ when $t = 0$ and $t = c$. 

		By formulae (\ref{for_sys_00}) and (\ref{for_sys_01}), when $t=0$, $\cosh |l_{CD}| = \cosh \frac{s}{2}\cosh \frac{t}{2} = \cosh \frac{s}{2}$. Therefore $2|l_{CD}| = s$. $\cosh |l_C| = \cosh s \cosh \frac{t}{2} \cosh (c-\frac{t}{2}) - \sinh \frac{t}{2} \sinh (c-\frac{t}{2}) =\cosh s \cosh c $. Therefore $2|l_{CD}| < |l_C| $ when $t = 0$. When $t = c$, 
		\begin{eqnarray*}
				\cosh 2|l_{CD}| &=& 2\cosh^2|l_{CD}| -1 \\
				&=& 2\cosh^2 \frac{s}{2} \cosh^2 \frac{t}{2} -1\\
				&=& 2\cosh^2 \frac{s}{2} \cosh^2 \frac{c}{2} -1\\
				&=&  \left( 2\cosh^2\frac{s}{2} -1 \right) \cosh^2 \frac{c}{2} + \cosh^2\frac{c}{2} -1\\
				&=& \cosh s \cosh^2 \frac{c}{2} + \sinh^2\frac{c}{2}. \\
				\cosh |l_{C}| &=& \cosh s \cosh \frac{t}{2} \cosh \left( c-\frac{t}{2} \right) - \sinh \frac{t}{2} \sinh \left( c-\frac{t}{2} \right) \\
				&=& \cosh s \cosh^2 \frac{c}{2} - \sinh^2\frac{c}{2}. 
		\end{eqnarray*}
		Therefore, when $t=c$, $2|l_{CD}| > |l_C|$. 

		For any fixed $c$, $|l_{CE}| \equiv c$ and $|l_{CE}(\Sigma)| \equiv 4c$. The systole of the surface is $\min(|l_{CE}(\Sigma)|,|l_{CD}(\Sigma)|,|l_{C}(\Sigma)| ) = 2\min (2|l_{CE}|,2|l_{CD}|,|l_{C}|)$.

		Figure \ref{fig_function} shows these lengths as functions about $t$. In this Figure, $(t_0, l_0)$ is the intersecting point of the graphs of the functions $2|l_{CD}|$ and $|l_{C}|$ about $t$. 
		
		(a) If $l_0> 2|l_{CE}| $, then $\forall t \in [0,c]$, $\min (2|l_{CE}|,2|l_{CD}|,|l_{C}|)\le 2|l_{CE}| = 2c$. When $t\in (t_1,t_2)$, $\min (2|l_{CE}|,2|l_{CD}|,|l_{C}|) = 2|l_{CE}| = 2c$. Here $t_1$, $t_2$ are the intersction point of $2|l_{CE}|,2|l_{CD}|$ and $2|l_{CE}|,|l_{C}|$ respectively (See Figure \ref{fig_function_1}). That is, the systole of the surface corresponding to $(c,t)$ is $4c$ when $t \in (t_1,t_2)$. Then $\forall (c,t)$ with $t\in (t_1,t_2)$, there is an $\varepsilon >0$, such that the systole of the surface corresponding to $(c+\varepsilon,t)$ is $4(c+\varepsilon)$. Therefoe if  $l_0> 2|l_{CE}|$, the corresponding surface cannot be the maximal surface. 
		\begin{figure}[htbp]
				\centering
		\subfigure[]{
				\includegraphics{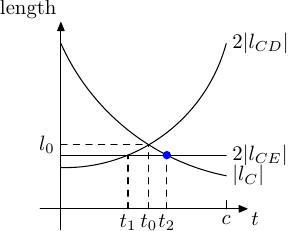}
				\label{fig_function_1}
		}
		\subfigure[]{
				\includegraphics{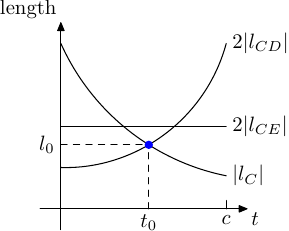}
				\label{fig_function_2}
		}
				\caption{}
				\label{fig_function}
		\end{figure}

		Therefore if a surface is the maximal $\Gamma(2,n)$ surface, then $l_0\le 2c$. 

		(b) If $2c > l_0$ then we prove the corresponding surface is not the maximal surface by changing the coordinate. 

		By the cut-and-paste described in the proof of Proposition \ref{prop_cuff_fake_not} (Figure \ref{fig_cut_paste}), we get another pentagon representation of the $S^2(2,2,2,n)$ induced by the $\Gamma(2,n)$ surface. 

\begin{figure}[htbp]
		\centering
		\subfigure[]{
		\includegraphics{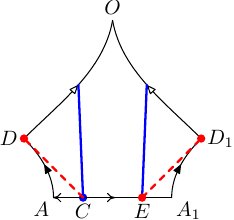}
		\label{fig_cut_paste_6}
}
		\subfigure[]{
		\includegraphics{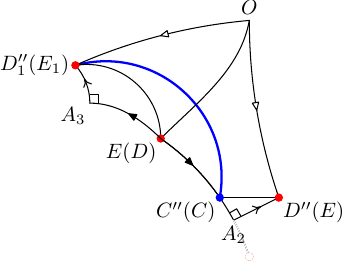}
		\label{fig_cut_paste_5}
}
		\caption{}
		\label{fig_cut_paste_reuse}
\end{figure}

		We pick the pentagon representations shown in Figure \ref{fig_cut_paste_1} and Figure \ref{fig_cut_paste_4}, see Figure \ref{fig_cut_paste_reuse}. To avoid confusion, we relabel the pentagon in Figure \ref{fig_cut_paste_5}. We define a new coordinate $(c'',t'')$ for pentagon in Figure \ref{fig_cut_paste_5}: $c'' = |C''E''|$, $t''=|C''D''|$. By the correspoindance of segments in the two pentagons, we have $|CD| = |C''E''|$ and $|CE| = |C''D''|$.

		Then by this observation, we have the following conclusion: 

		For a $\Gamma(2,n)$ surface represented by the pair $(c,t)$, there exists a pair $(c'',t'')$ such that:

		(1) The surfaces represented by $(c,t)$ and $(c'',t'')$ are isometric. 

		(2) $|l_{CD}| = |l_{C''E''}|$, $|l_{CE}| = |l_{C''D''}|$ and $|l_{C}| = |l_{C''}|$. 

		If $2c>l_0$, then for fixed $c$, the maximal systole about $t$ is realized by the coordinate $(c,t_0)$, the blue point $P$ in Figure \ref{fig_function_2}. At this point, $|l_C| = 2|l_{CD}|$ and $2|l_{CE}| > |l_{C}|$ (See Figure \ref{fig_function_2}). Then by the conclusion, we have there is $(c'',t'')$ corresponding to the isometric surface as $(c,t)$ such that $|l''_C| = 2|l''_{CE}|$ and $2|l''_{CD}| > |l''_{C}|$ (the blue point $P''$ in Figure \ref{fig_function_1}). By the proof in (a), if $2c>l_0$, then the surface is not maximal. 

		By (a) and (b), if the surface is maximal then $l_0=2c$, It implies that $|l_C| = 2|l_{CD}| = 2|l_{CE}|$, namely $|l_C(\Sigma)| = |l_{CD}(\Sigma)| = |l_{CE}(\Sigma)|$.

\end{proof}

\subsection{Final calculation}

Finally we calculate the systole of the surface with $|l_C(\Sigma)| = |l_{CD}(\Sigma)| = |l_{CE}(\Sigma)|$. To calculate the length, we obtain a subsurface with the signature $(1,2)$ by cutting along the red curves $c_0, c'_0$ in Figure \ref{fig_exp_1_2}. $C_1$, $C_2$, $C_3$ and $C_4$ are branch points of the branch cover $\pi$. By the proof of Theorem \ref{thm_equ}, in the $\Sigma_{(1,2)}$, the curve connecting $C_1$, $C_2$ is $l_{CE}(\Sigma)$, the curve connecting $C_1$, $C_3$ is $l_{C}(\Sigma)$, the curve connecting $C_2$, $C_3$ is $l_{CD}(\Sigma)$. 

\begin{figure}[htbp]
		\centering
		\includegraphics{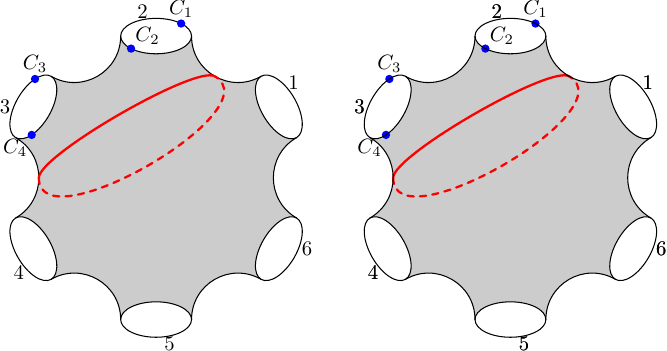}
		\caption{}
		\label{fig_exp_1_2}
\end{figure}

Next, we represent the length of $\partial \Sigma_{(1,2)}$ by $c$ and $s$. The seams and cuffs of the $\Gamma(2,n)$ surface that intersect $\Sigma_{(1,2)}$ cut $\Sigma_{(1,2)}$ into four equal hexagons(See Figure \ref{fig_sigma12_1}). In Figure \ref{fig_sigma12_1}, $A_2A_1H_2H_1A_4A_3$ is one of the hexagons. This hexagon is a right-angled hexagon. It is clear the angles are right at the four vertices $A_1$, $A_2$, $A_3$ and $A_4$ by the definition of cuffs and seams. For $H_1$ and $H_2$, we consider the $n$-holed spheres of the $\Gamma(2,n)$ surface. If cutting all the seams of one $n$-holed sphere, we get two isometric $2n$-polygon with order $n$ rotations. In each polygon, we pick the common perpendicular between two nearest non-neighboring seams. Two such curves, each from a polygon, forms a simple closed curve. This curve is one component of $\partial\Sigma_{(1,2)}$. 
\begin{figure}[htbp]
		\centering
		\includegraphics{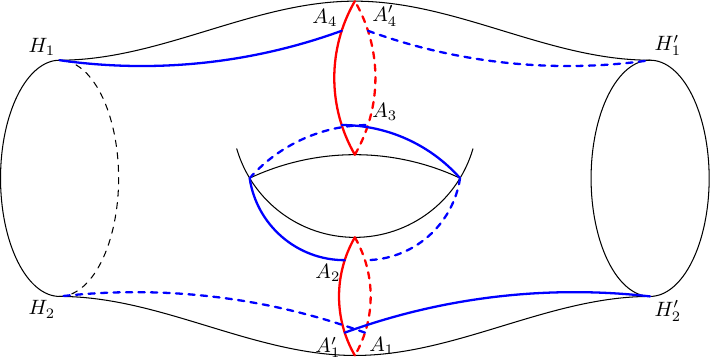}
		\caption{The red curves are cuffs. The blue curves are seams. }
		\label{fig_sigma12_1}
\end{figure}

In the right-angled hexagon $A_2A_1H_2H_1A_4A_3$, $|A_1A_2| = |A_3A_4| = c$, $|A_2A_3| = s$, then by Formula 2.4.1(i) in \cite[p. 454]{buser2010geometry}, 
\begin{align}
	\label{for_partial_sigma_12}	\cosh |H_1H_2| &= \sinh |A_1A_2|\sinh |A_3A_4|\cosh |A_2A_3| - \cosh |A_1A_2|\cosh |A_3A_4| \\
		\nonumber&= \sinh^2c\cosh s - \cosh^2c. 
\end{align}

We pick another set of curves to cut $\Sigma_{(1,2)}$ to obtain the length of the systole. We pick two common perpendicular segmets between the two boundary components of $\Sigma_{(1,2)}$. One of the segmet is homotopic to the broken segment $H_1A_3A_3'H_1'$ (see Figure \ref{fig_sigma12_1}) relative to the boundary, the other is homotopic to the broken segment $H_2A_1A_1'H_2'$ relative to the boundary. These two segments are $H_3H_3'$ and $H_4H_4'$ in Figure \ref{fig_sigma12_2} respectively. In Figure \ref{fig_sigma12_2}, by the definition of $t$ in \ref{sec_geo_orbi} and the symmetry of $\Sigma_{(1,2)}$, the mid-point of $H_3H_3'$ and $H_4H_4'$ are the branch points $C_1$ and $C_4$ respectively. 

\begin{figure}[htbp]
		\centering
		\includegraphics{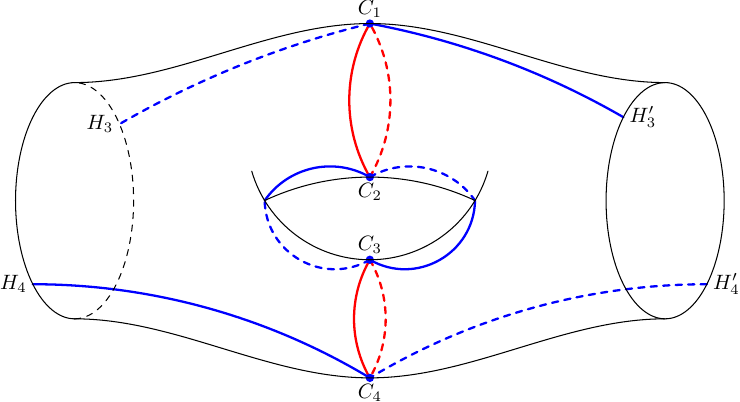}
		\caption{}
		\label{fig_sigma12_2}
\end{figure}

Then by cutting along $H_3H_3'$ and $H_4H_4'$, we get a surface with the topology of annulus (see Figure \ref{fig_sigma12_3}). The unique non-trivial closed geodesic is illustrated by the segment $H_5H_6$.  The common perpendiculars between $H_5H_6$ and $H_iH_i'$ ($i= 1,2$) meet the mid-points $C_1$ or $C_4$ since all such common perpendiculars have the same length by the symmetry of $\Sigma_{(1,2)}$, then by the sine law of right-angle hexagon (Formula 2.4.1 (ii) in \cite[ p 454]{buser2010geometry}), the common perpendiculars meet the mid-points of $H_iH_i'$ ($i=1,2$). 

\begin{figure}[htbp]
		\centering
		\includegraphics{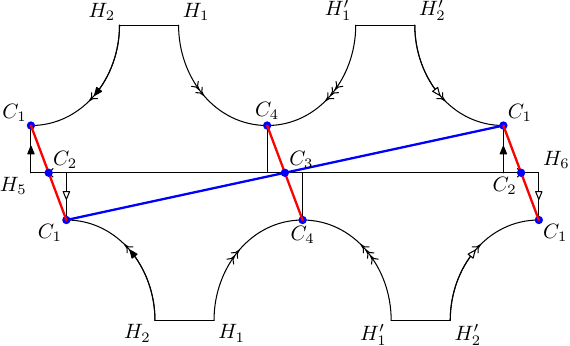}
		\caption{}
		\label{fig_sigma12_3}
\end{figure}

Then we describe the curve $l_{CD}(\Sigma)$, $l_{CE}(\Sigma)$ and $l_{C}(\Sigma)$ in the annulus. We recall that $l_{CD}(\Sigma)$ passes $C_2$ and $C_3$, $l_{CE}(\Sigma)$ passes $C_1C_2$ or $C_3C_4$ and $l_{C}(\Sigma)$ passes $C_1C_3$ or $C_2C_4$. 

The curve $l_{CD}(\Sigma)$ (the curve meeting $C_2$ and $C_3$) does not touch $H_1H_1'$ or $H_2H_2'$. Therefore it is the unique non-trivial closed geodesic in the annulus, namely the curve in Figure \ref{fig_sigma12_3} corresponding to $H_5H_6$. In Figure \ref{fig_sigma12_3}, the red segments ($C_1C_2C_1$ and $C_4C_3C_4$) are cuffs ($l_{CE}(\Sigma)$). Therefore the intersction between cuffs and $H_5H_6$ are the branch points $C_2$ and $C_3$. The curve $l_C(\Sigma)$ passes $C_1C_3$ or $C_2C_4$. Then in Figure \ref{fig_sigma12_3}, the blue curve connecting $C_1C_3C_1$ is one of the $l_C(\Sigma)$ curves. 

By Proposition \ref{prop_equal}, $|l_{C}(\Sigma_0)| = |l_{CD}(\Sigma_0)| = |l_{CE}(\Sigma_0)|$. Then in Figure \ref{fig_sigma12_3}, the triangle $\triangle C_1C_2C_3$ is a equilateral triangle, because the edges $C_1C_2$, $C_2C_3$ and $C_3C_1$ are the half of $l_{CE}(\Sigma)$, $l_{CD}(\Sigma)$ and $l_{C}(\Sigma)$ respectively. When the triangle is equilateral, the shape of the annulus is shown in Figure \ref{fig_sigma12_4}. 

\begin{figure}[htbp]
		\centering
		\includegraphics{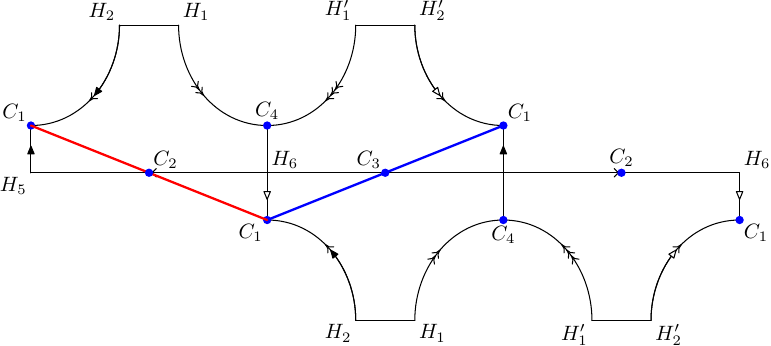}
		\caption{}
		\label{fig_sigma12_4}
\end{figure}

We assume in Figure \ref{fig_sigma12_4}, $h = |C_1H_5| = |C_4H_6|$, $k = |C_1C_2| = |C_2C_3| = |C_3C_1|$. Then $|C_2H_6| = |C_3H_6| = k/2$. Then we have the following two formulae: 

In the hexagon $H_1H_2C_1H_5H_6C_4$, by formula 2.4.1(i) in \cite[ p 454]{buser2010geometry} we have:
\begin{align}
		\nonumber\cosh |H_1H_2| &= \sinh |C_1H_5| \sinh |C_4H_6|\cosh |H_5H_6| - \cosh |C_1H_5| \cosh |C_4H_6| \\
		\nonumber\cosh l &= \sinh^2 h \cosh k - \cosh^2 h \\
		\cosh^2 h &= \frac{\cosh l + \cosh k}{\cosh k -1}. \label{for_h_by_l}
\end{align}

Then in the triangle $\triangle C_1H_6C_3$, by the hyperbolic cosine law (Formula 2.2.2 (i) in \cite[p. 454]{buser2010geometry} ), we have 
\begin{align}
		\nonumber\cosh |C_1C_3| &= \cosh |C_1H_6| \cosh |C_3H_6| \\
		\cosh k &= \cosh h \cosh \frac{k}{2}. \label{for_k_and_h}
\end{align}

Finally we can get the formula for $\cosh k$:

\begin{align}
		\nonumber\frac{\cosh^2 k}{\cosh^2 \frac{k}{2}} &= \cosh h \,\,\,\, \text{by (\ref{for_k_and_h})} \\
		&= \frac{\cosh l + \cosh k}{\cosh k -1}.\,\,\,\, \text{by (\ref{for_h_by_l})} 
		\label{for_eli_h}
\end{align}
 For convenience, we assume $K = \cosh k$, Then 
 \[
		 \frac{\cosh^2 k}{\cosh^2 \frac{k}{2}} = \frac{2\cosh^2 k}{\cosh k + 1} = \frac{2K^2}{K+1}.
 \]

 We use it on (\ref{for_eli_h}): 
 \begin{align}
		 \nonumber\frac{2K^2}{K+1} &= \frac{K+\cosh l}{K-1} \\
		 \nonumber&= \frac{K+ \sinh^2c\cosh s - \cosh^2c}{K-1} \,\,\,\,\text{by (\ref{for_partial_sigma_12})} \\
		 \nonumber&= \frac{(K^2-1)\cosh s -K^2+K}{K-1} \,\,\,\,c=k\text{ by definition} \\
		 \nonumber&= (K+1) \cosh s - K \\
		 \nonumber&= (K+1) (2\sinh^2 \frac{s}{2} +1 )- K \\
		 \nonumber&= (K+1) \left( 2\frac{\cos^2\frac{\pi}{n}}{\sinh^2\frac{c}{2}} +1 \right) -K\,\,\,\,\text{by (\ref{for_cs})} \\
		 &= (K+1) \left( \frac{4\cos^2\frac{\pi}{n}}{K-1} +1 \right) - K. 
		 \label{for_sys_equ}
 \end{align}

 Then from (\ref{for_sys_equ}), we have 
 \[
		 2K^3-3K^2+1-4\cos^2\frac{\pi}{n}(K+1)^2 =0
 \]

 The unique real solution of this equation is: 
\begin{eqnarray*}
		K &=& \sqrt[3]{\frac{1}{216}L^3 +\frac{1}{8} L^2 + \frac{5}{8} L - \frac{1}{8} + \sqrt{\frac{1}{108}L(L^2+18L+27)} } \\
		& & + \sqrt[3]{\frac{1}{216}L^3 +\frac{1}{8} L^2 + \frac{5}{8} L - \frac{1}{8} - \sqrt{\frac{1}{108}L(L^2+18L+27)} } \\
		& & + \frac{L+3}{6}.
\end{eqnarray*}
and  $L= 4\cos^2 \frac{\pi}{n}$. 

 Now we calculate the coordinate $(c,t)$ when the systole is maximal. 

 It is clear that $c=\arccosh K$. Then we calculate $t$ using (\ref{for_sys_00}). 
\begin{eqnarray*}
		\cosh \frac{t}{2} &=& \frac{\cosh \frac{c}{2}}{\cosh \frac{s}{2}} \\
		&=& \frac{\cosh^2 \frac{c}{2}}{\cos \frac{\pi}{n}} \,\,\,\,\text{by (\ref{for_cs})} \\
		&=& \frac{\cosh c + 1}{2 \cos \frac{\pi}{n}} \\
		&=& \frac{K + 1}{2 \cos \frac{\pi}{n}}. 
\end{eqnarray*}

 Therefore we get the Theorem:
 \begin{theorem}
		 The maximal systole of the $\Gamma(2,n)$ surface is 
		 \[
				 2\arccosh K. 
		 \]
Here 
\begin{eqnarray*}
   	 K &=& \frac{1}{6} \left( T^3+27T^2+12\sqrt{3}\sqrt{T^3+18T^2-27T} +135T-27 \right)^{\frac{1}{3}} + \\
   	 & &\frac{T^2+18T+9}{6\left( T^3+27T^2 + 12\sqrt{3}\sqrt{T^3+18T^2-27T} +135T-27  \right)}  + \frac{T+3}{6}, 
\end{eqnarray*}

and  $T= 4\cos^2 \frac{\pi}{n}$. 

 The maximal systole is obtained when 
 \[
		 (c,t) = \left ( \arccosh K, 2\arccosh \frac{K + 1}{2 \cos \frac{\pi}{n}} \right ).  
 \]

		 \label{thm_main}
 \end{theorem}

\bibliographystyle{alpha}

\end{document}